\documentclass[a4paper, 11pt]{article}

\usepackage{amsmath,amssymb,amsthm}
\usepackage[margin=1.45in]{geometry}
\usepackage{url}
\usepackage{comment}
\usepackage{xcolor}

\usepackage[
        hypertexnames=false,
        hyperindex,
        breaklinks=true,
        bookmarks=false,
        colorlinks,
        linkcolor=blue,
        citecolor=purple,
        urlcolor=purple,
]{hyperref}

\newtheorem{theorem}{Theorem}[section]
\newtheorem{proposition}[theorem]{Proposition}
\newtheorem{corollary}[theorem]{Corollary}
\newtheorem{lemma}[theorem]{Lemma}
\theoremstyle{definition}
\newtheorem{definition}[theorem]{Definition}
\newtheorem{remark}[theorem]{Remark}
\theoremstyle{plain}
\newtheorem{conjecture}[theorem]{Conjecture}
\newtheorem{fact}[theorem]{Fact}

\newcommand{\F}{\mathbb{F}}
\newcommand{\Q}{\mathbb{Q}}
\newcommand{\Ql}{\mathbb{Q}_\ell}
\newcommand{\Z}{\mathbb{Z}}

\newcommand{\PGU}{\mathrm{PGU}}
\newcommand{\GU}{\mathrm{GU}}

\newcommand{\Herm}{\mathcal H}
\newcommand{\Frob}{\mathrm{Fr}}
\newcommand{\Wst}{W_{r}}
\newcommand{\Ur}{U_{r}}
\newcommand{\cupp}{\mathbin{\smallsmile}}
\newcommand{\capp}{\mathbin{\smallfrown}}
\newcommand{\chern}{\mathrm{c}}
\newcommand{\disc}{\operatorname{disc}}
\newcommand{\WO}{\mathcal O}

\title{An explicit form of the standard conjecture of Hodge type for Hermitian varieties, with $p$-adic applications}
\author{Shushi Harashita\thanks{Faculty of Environment and
Information Sciences, Yokohama National University,
79-7 Tokiwadai, Hodogaya-ku, Yokohama 240-8501, Japan.
Email: \texttt{harashita-shushi-nr@ynu.ac.jp}}}
\date{}

\begin{document}
\maketitle

\begin{abstract}
Let $X$ be a nonsingular Hermitian variety of even dimension over a finite field,
and let $\Lambda$ be the lattice spanned by the classes of its maximal totally
isotropic linear subspaces defined over the base field. By work of Dummigan, Dummigan--Tiep and Shimada,
these classes span the middle cohomology, and the intersection form on the
primitive part is definite, as predicted by the standard conjecture of Hodge
type. We give an elementary proof, which avoids the representation theory of
finite unitary groups and uses only the dual polar graph attached to the Hermitian form. It shows that the intersection form on the primitive part is an explicit
positive multiple of a Euclidean inner product.

The advantage of the new proof is that it also controls $\Lambda$ $p$-adically. If
$X$ is defined over $\F_{q^2}$ and has dimension $2m$, every $p$-adic elementary
divisor of the intersection matrix of these subspaces divides $q^m$. Consequently
$q^m$ times every algebraic cycle of codimension $m$ is numerically equivalent to an
integral combination of them. As a by-product, the standard conjecture of Hodge
type holds for varieties finitely covered by Hermitian varieties, and the
$p$-adic bound descends along finite morphisms of degree prime to $p$.
\end{abstract}

\section{Introduction}
\label{sec:intro}

Let $Y$ be a smooth projective variety of dimension $n$ over an algebraically
closed field, $\ell$ a prime different from the characteristic, $L$ the class
of an ample divisor, and $\mathcal A^i(Y)\subset H^{2i}(Y,\Ql(i))$ the $\Q$-span of
the classes of codimension-$i$ algebraic cycles. Grothendieck's \emph{standard
conjecture of Hodge type} $\mathrm{Hdg}(Y)$ asserts that for $2i\le n$ the form
$(x,y)\mapsto(-1)^i\deg(x\cupp y\cupp L^{n-2i})$ is positive definite on the
primitive part of $\mathcal A^i(Y)$ (\S\ref{subsec:hodge-survey}). Over $\mathbb C$
this follows from Hodge theory. In characteristic $p$ it is open in general. It
is known for surfaces \cite[Thm.~V.1.9]{Hartshorne}, for abelian fourfolds
\cite{Ancona}, for self-products of K3 surfaces \cite{IIK}, and for the
configurations of linear cycles studied by Ito \cite{Ito}.

For the Hermitian varieties $\Herm(M,q^2)\colon\sum_{i=0}^Mx_i^{q+1}=0$ the
conjecture is known. In odd dimension it is trivial
(Proposition~\ref{prop:belowmiddle}(i)). In even dimension it was observed by
Dummigan \cite{Dummigan1997} and proved in full by Dummigan and Tiep
\cite[\S8]{DT1999}. The finite unitary group $\mathrm U_{M+1}(q)$ of the Hermitian
form acts on $\Herm(M,q^2)$. For $M$ odd, they used that the primitive part of the
middle cohomology is irreducible under this group, a fact stated by Tate and
Thompson \cite[p.~102]{Tate1965}; it is a constituent of the Weil representation
\cite{DT1999}, and irreducible even over $\mathbb R$.
Therefore the invariant symmetric bilinear forms on it are unique up to a scalar
\cite[Prop.~2.1]{Gross1990}, so the intersection form is definite; its sign was determined with
explicit linear cycles. Shimada \cite[Thm.~1.1, Cor.~1.2]{Shimada2001}
gave another proof by computing the lattice of generators. He proved that
their classes span the middle cohomology (recovering the supersingularity due to
Tate \cite{Tate1965} and Shioda--Katsura \cite{SK1979}), that the intersection form
has the predicted signature, and that the discriminant is a power of $p$.

The first part of this paper gives an \emph{elementary proof} of the spanning and
of the conjecture (Theorems~\ref{thm:main-intro} and~\ref{thm:tate-intro}). It
uses neither the classification of characters of finite unitary groups nor
Deligne--Lusztig theory, only the dual polar graph. The second part uses the same structure to study the lattice of
generators $p$-adically (\S\ref{sec:gentrace}). Some of the facts
we use are already in \cite{Shimada2001}. The intersection numbers of generators
are \cite[(4.1)]{Shimada2001}; we include their derivation for the reader's
convenience. For the counts of generators and the Betti numbers we give
references.

Before stating the result we fix the conventions used throughout.
Here $p$ is a prime, $q=p^{\nu}$ a power of it, $r\ge2$ an integer,
$K=\overline{\F}_p$, and $\ell\ne p$ an auxiliary prime. From
\S\ref{sec:setup} to \S\ref{sec:hodge} we write
$X=\Herm(2r-1,q^{2})$; other varieties occur only in
\S\ref{sec:fermat}. We write
$d:=q+1$ and
\[
  N_s:=\prod_{i=1}^{s}(q^{2i-1}+1)\quad (N_0=1).
\]
Cohomology means $\ell$-adic \'etale cohomology of the base change to
$K$; we abbreviate $H^{m}(X,\Ql):=H^{m}(X_K,\Ql)$, always displaying
the coefficients, and insert Tate twists where they matter. Dimensions of projective linear
subspaces are projective unless a subscript indicates otherwise: for a
projective subspace $P\subset\mathbb P^{2r-1}$ we write $\widehat P\subset
K^{2r}$ for its cone, the corresponding linear subspace of the ambient
$2r$-dimensional $K$-vector space, so that $\dim_K\widehat P=\dim P+1$;
if $P$ is $\F_{q^2}$-rational we write $\widetilde
P\subset\F_{q^2}^{\,2r}$ for the corresponding $\F_{q^2}$-subspace, so
that $\widehat P=\widetilde P\otimes_{\F_{q^2}}K$ and
$\dim_{\F_{q^2}}\widetilde P=\dim_K\widehat P$. By
convention $\dim\emptyset=-1$.

Since $\dim X=2r-2$ is even, the
middle cohomology is $H^{2r-2}(X_{\overline\F_p},\Ql(r-1))$; write
$\langle\,,\,\rangle$ for the Poincar\'e (cup-product) pairing on it,
$L$ for the hyperplane class, and
\[
  \mathcal P^{2r-2}(X)\;:=\;\bigl(L^{r-1}\bigr)^{\perp}
  \;=\;\ker\Bigl(L:H^{2r-2}\bigl(X,\Ql(r-1)\bigr)\to H^{2r}\bigl(X,\Ql(r)\bigr)\Bigr)
\]
for the primitive part.

Next we describe the cycles that will span this
cohomology. Let $\bar x:=x^{q}$ denote the nontrivial automorphism of
$\F_{q^{2}}/\F_{q}$ and let
\begin{equation}
\label{eq:hermform-intro}
  \langle x,y\rangle_{\mathrm H}\;:=\;\sum_{i=0}^{2r-1}x_i\,\bar y_i
  \qquad\bigl(x,y\in\F_{q^{2}}^{\,2r}\bigr)
\end{equation}
be the standard nondegenerate Hermitian form relative to it. Since
$\langle x,x\rangle_{\mathrm H}=\sum_i x_i^{\,q+1}$, the variety $X$ is
exactly the locus of isotropic points of
$\langle\,,\,\rangle_{\mathrm H}$ in $\mathbb P^{2r-1}$. A totally
isotropic $\F_{q^{2}}$-subspace of $\F_{q^{2}}^{\,2r}$ has dimension at
most $r$, and a \emph{generator} of $X$ is an $\F_{q^2}$-rational projective linear subspace
$P\subset\mathbb P^{2r-1}$ whose $\F_{q^{2}}$-subspace $\widetilde P$ is a totally
isotropic subspace of $\F_{q^{2}}^{\,2r}$ of the maximal dimension $r$; thus
$P\cong\mathbb P^{r-1}$ is an $\F_{q^{2}}$-rational linear subvariety of
$X$ of exactly half its dimension. We write $\mathcal G$ for the set of
generators (Definition~\ref{def:generator}).
It is a finite set, of cardinality $N_r$
(Proposition~\ref{prop:generators}), and every $P\in
\mathcal G$ has a cycle class $[P]$ in the middle cohomology. The span
of these classes, and its primitive part
\[
  \mathcal P^{2r-2}(X)_{\Q}\;:=\;\mathcal P^{2r-2}(X)\,\cap\,
  \mathrm{span}_{\Q}\bigl\{[P]:P\in\mathcal G\bigr\},
\]
are $\Q$-vector spaces on which $\langle\,,\,\rangle$ takes rational
values, so that it makes sense to ask whether a form is definite
there. In fact the generators already account for all algebraic
cycles in this degree, so that
$\mathcal P^{2r-2}(X)_{\Q}=\mathcal A^{r-1}_{\mathrm{prim}}(X)$
\textup{(}Proposition~\ref{lem:Qstructure}\textup{)}\textup{;} the
positivity statement in Theorem~\ref{thm:main-intro} below is
therefore exactly $\mathrm{Hdg}(X)$ in the sense of
Conjecture~\ref{conj:hdg}. Finally, let $\Q^{\mathcal G}$ be the euclidean space with
orthonormal basis $\{e_P\}_{P\in\mathcal G}$, that is, with the inner
product
\[
  \bigl\langle u,u'\bigr\rangle_{\mathrm{std}}
  \;:=\;\sum_{P\in\mathcal G}u_P\,u'_P,
  \qquad u=\sum_P u_Pe_P,\quad u'=\sum_Pu'_Pe_P .
\]
(We reserve $\langle\,,\,\rangle$, without subscript, for the
Poincar\'e pairing of \eqref{eq:pairingdef} below, and write
$\langle\,,\,\rangle_{\mathrm{std}}$ for this standard euclidean
inner product on $\Q^{\mathcal G}$ or $\Ql^{\mathcal G}$, to avoid any
clash with the notation $(u,v)$ for an ordered pair.)

\begin{theorem}[Main theorem; proved in \S\ref{sec:hodge}]
\label{thm:main-intro}
Let $r\ge2$, let $q$ be any prime power and set $X=\Herm(2r-1,q^2)$,
the Hermitian variety of even dimension $2r-2$ \textup{(}the
odd-dimensional Hermitian varieties $\Herm(2r,q^2)$ being covered by
Proposition~\textup{\ref{prop:belowmiddle}(i)}\textup{)}.
Then the standard conjecture of Hodge type holds for $X$\textup{:} the
symmetric bilinear form
\[
  (x,y)\;\longmapsto\;(-1)^{r-1}\langle x,y\rangle
\]
is positive definite on $\mathcal P^{2r-2}(X)_{\Q}$.

In fact more is true: the form is, up to one constant factor, a
euclidean inner product. Precisely, there are a linear subspace
$\Ur\subseteq\Q^{\mathcal G}$ and an isomorphism of $\Q$-vector
spaces
\[
  \psi\;:\;\Ur\;\xrightarrow{\ \sim\ }\;\mathcal P^{2r-2}(X)_{\Q}
\]
such that
\[
  (-1)^{r-1}\bigl\langle\psi(u),\psi(u')\bigr\rangle
  \;=\;q^{\,r-1}N_{r-1}\cdot\bigl\langle u,u'\bigr\rangle_{\mathrm{std}}
  \qquad\text{for all }u,u'\in\Ur .
\]
Here $\Ur$ is an eigenspace of an explicit integer
matrix, constructed in \S\textup{\ref{sec:decomp}}.
\end{theorem}

Positivity is \cite[\S8]{DT1999} and \cite[Cor.~1.2]{Shimada2001}. That the form is
proportional to a Euclidean one also follows from the irreducibility used by
Dummigan and Tiep, via Schur's lemma. The new points are the elementary proof,
the explicit constant $q^{r-1}N_{r-1}$, and the fact that the argument rests on
an identity of integer matrices, which will give $p$-adic information. The same computation gives, with no appeal to
unirationality, the following statement, which is also contained in
\cite{Tate1965,SK1979,Shimada2001}:

\begin{theorem}[proved in \S\ref{sec:tate-cor}]
\label{thm:tate-intro}
For every $r\ge2$ and every prime power $q$, the geometric Frobenius
$\Frob_{q^2}$ acts on $H^{2r-2}(X_{\overline\F_p},\Ql)$ by the scalar
$q^{2r-2}$, and $H^{2k}(X_{\overline\F_p},\Ql(k))$ is spanned by
algebraic cycle classes for every $k$. In particular the Tate
conjecture holds for $X$ unconditionally.
\end{theorem}

\medskip
\noindent\emph{Idea of the proof.} Two generators meet cleanly, and by the excess
intersection formula $P\cdot P'=f(j)=(1-(-q)^{j+1})/(q+1)$ with
$j=\dim(P\cap P')$ (\S\ref{subsec:intnum}). For generators in special position
with respect to the coordinates this is \cite[Lemma~9.1]{DT1999}; see also
Shioda's remark quoted there. So the Gram matrix $M$ lies in the
Bose--Mesner algebra of the dual polar scheme of type ${}^2A_{2r-1}(q)$. In fact
$M=(J-M')/(q+1)$, where $J$ is the all-ones matrix and
$M'=\bigl((-q)^{\dim\widetilde P\cap\widetilde P'}\bigr)$. One identity among the
parameters of the scheme gives $AM'=\theta_rM'$ for the adjacency matrix $A$
(\S\ref{subsec:key}). Hence $M'$ is a multiple of the orthogonal projection onto
the eigenspace $\Wst$, and $M$ has exactly two nonzero eigenvalues. Comparing
$\dim\Wst$ with the Betti number forces the generators to span
(\S\ref{subsec:span}). The form on $\Wst$ is multiplied by
$(-1)^{r-1}q^{r-1}N_{r-1}>0$.

\subsection*{The lattice of generators, $p$-adically}

Let $\Lambda\subset\mathrm N^{r-1}(X)$ be the lattice spanned by the generators
inside the group of cycles modulo numerical equivalence. By \cite{Shimada2001},
$\mathrm N^{r-1}(X)/\Lambda$ is a finite $p$-group, and whether it vanishes is
open in general. The diagonal torus $S=\mu_{q+1}^{2r}/\mu_{q+1}$ acts on
$\mathcal G$, and $\Lambda\otimes W(\F_{q^2})$ splits into rank-one components
$\Lambda_a$ indexed by characters. Let $e_\Lambda(a)$ be the valuation of the
pairing $\Lambda_a\times\Lambda_{-a}$. These are exactly the exponents of the
$p$-adic elementary divisors of the Gram matrix. From $M'^2=\eta_rM'$ we derive
a trace identity for the torus-Fourier coefficients of the intersection form
(Proposition~\ref{prop:gentrace}). This information is not accessible through
Schur's lemma, since the relevant modules need not be irreducible modulo $p$. As consequences (Corollary~\ref{cor:gentrace}):
\begin{itemize}
\item $e_\Lambda(a)\le\nu(r-1)$ for every $a$, where $q=p^\nu$; so the discriminant
  group of $\Lambda$ is killed by $q^{r-1}$;
\item $q^{r-1}\,\mathrm N^{r-1}(X)\subset\Lambda$;
\item $v_p[\mathrm N^{r-1}(X):\Lambda]\le\frac12\nu(r-1)(\beta_{2r-2}-1)$.
\end{itemize}

\subsection*{Finite covers}

The standard conjecture of Hodge type descends along finite surjective
morphisms (Lemma~\ref{lem:descent}), which gives:

\begin{theorem}[proved in \S\ref{sec:fermat}]
\label{thm:fermat-intro}
Let $Z$ be a nonsingular projective variety of even dimension over
$\overline{\F}_p$ admitting a finite surjective morphism from a
Hermitian variety --- for instance a nonsingular finite quotient of
one. Then the standard conjecture of Hodge type holds for $Z$.
\end{theorem}

For example, let $X_N^n\colon\sum_{i=0}^{n+1}x_i^{\,N}=0$ be the Fermat variety of
degree $N$ ($p\nmid N$). If $N\mid q+1$ and $e=(q+1)/N$, then $X_N^{2m}$ is the
quotient of $\Herm(2m+1,q^2)$ by $\mu_e^{2m+2}/\mu_e$ via $[x_i]\mapsto[x_i^{\,e}]$, so
Theorem~\ref{thm:fermat-intro} applies.

\begin{corollary}[proved in \S\ref{sec:fermat}]
\label{cor:fermat-intro}
The standard conjecture of Hodge type holds for the Fermat variety
$X_N^{2m}$ over $\overline{\F}_p$ whenever
\begin{equation}
\label{eq:sscongruence}
  p^{\nu}\equiv-1\pmod N\qquad\text{for some }\nu\ge1 ,
\end{equation}
which is precisely the condition for a suitable $q$ to exist. By
Theorem~II of \cite{SK1979}, for $N\ge4$ condition
\eqref{eq:sscongruence} characterises the \emph{supersingular} Fermat
varieties of even dimension, supersingular meaning here that all even
cohomology is spanned by algebraic cycle classes
\textup{(}Definition~\ref{def:ss}\textup{)}. In
odd dimension the conjecture holds for $X_N^{n}$ with no hypothesis on
$N$ beyond $p\nmid N$, by Proposition~\ref{prop:belowmiddle}\textup{(i)}.
\end{corollary}

The $p$-adic bounds descend as well. For any finite surjective morphism
$\pi\colon X\to Z$ of degree prime to $p$, with $2m=\dim Z$, the lattice
$\Lambda_Z$ spanned by the images $\pi_*[P]$ of the generators satisfies
$q^m\,\mathrm N^m(Z)\subset\Lambda_Z\otimes\Z_{(p)}$, and $\Lambda_Z$ is unimodular at
the primes $\ell\ne p$ not dividing $\deg\pi$
(Proposition~\ref{prop:pdescent}). For quotients by subgroups of
$\PGU_{2m+2}(q)$, such as the supersingular Fermat varieties, the Gram matrix of
$\Lambda_Z$ is obtained by summing over orbits (Corollary~\ref{cor:pdescent}).

\section{Preliminaries}
\label{sec:prelim}

This section is a survey. It fixes the statements of the two
conjectures involved and collects, with precise references, the
external results used later. 

\subsection{The Tate conjecture}
\label{subsec:tate-survey}

Let $Y$ be a smooth projective variety of dimension $n$ over a finite
field $\F_{q_0}$ of characteristic $p$, let $\ell\ne p$, and let
$\Frob_{q_0}$ denote the geometric Frobenius acting on
$H^{2k}(Y_{\overline{\F}_{q_0}},\Ql(k))$. Let
$Z^{k}(Y)$ be the group of algebraic cycles of codimension $k$ on $Y$
\emph{defined over the base field $\F_{q_0}$} and
\[
  \mathrm{cl}^{k}\;:\;Z^{k}(Y)\otimes\Ql\;\longrightarrow\;
  H^{2k}\bigl(Y_{\overline{\F}_{q_0}},\Ql(k)\bigr)
\]
the cycle class map.

\begin{conjecture}[Tate \cite{Tate1965}]
\label{conj:tate}
The image of $\mathrm{cl}^{k}$ equals the space of Frobenius
invariants
\[
  H^{2k}
  \bigl(Y_{\overline{\F}_{q_0}},\Ql(k)\bigr)^{\Frob_{q_0}=1}.
\]
\end{conjecture}

\subsection{The standard conjecture of Hodge type}
\label{subsec:hodge-survey}

Let $Y$ be smooth projective of dimension $n$ over an algebraically
closed field, $L$ the class of an ample divisor, and
\[
  \mathcal P^{2i}(Y)\;:=\;\ker\Bigl(L^{\,n-2i+1}:H^{2i}\bigl(Y,\Ql(i)\bigr)
  \to H^{2n-2i+2}\bigl(Y,\Ql(n-i+1)\bigr)\Bigr)
\]
the primitive part in degree $2i\le n$. Let $\mathcal A^{i}(Y)\subset
H^{2i}(Y,\Ql(i))$ denote the $\Q$-span of the algebraic cycle classes
of codimension $i$, and $\mathcal A^{i}_{\mathrm{prim}}(Y)$ its
primitive part. The pairing
$(x,y)\mapsto\deg(x\cupp y\cupp L^{\,n-2i})$ takes rational values on
$\mathcal A^{i}(Y)$, because it does so on classes of algebraic
cycles; the $\Q$-span is taken precisely so that the positivity
assertion below makes sense.

\begin{conjecture}[Grothendieck; standard conjecture of Hodge type
$\mathrm{Hdg}(Y)$, see {\cite[\S3]{Kleiman}}]
\label{conj:hdg}
For $2i\le n$ the symmetric bilinear form
\[
  (x,y)\;\longmapsto\;(-1)^{i}\deg\bigl(x\cupp y\cupp L^{\,n-2i}\bigr)
\]
is positive definite on $\mathcal A^{i}_{\mathrm{prim}}(Y)$.
\end{conjecture}

For a hypersurface, all degrees strictly below the middle are empty of
content. We record this once, since it is what makes the
odd-dimensional case of our results trivial and what disposes of the
non-middle degrees in the even-dimensional case; in even
dimension we also record the resulting orthogonal decomposition of the
middle cohomology.

\begin{proposition}[Hodge type below the middle degree; Lefschetz
decomposition]
\label{prop:belowmiddle}
Let $Y\subset\mathbb P^{n+1}$ be a nonsingular hypersurface of
dimension $n\ge1$ over an algebraically closed field, and let $L$ be
the hyperplane class. In \textup{(ii)--(v)} assume moreover that
$n=2m$ is even, and write $\langle a,b\rangle:=\deg(a\cupp b)$ for the
Poincar\'e pairing on $H^{n}(Y,\Ql(m))$. Then\textup{:}
\begin{enumerate}
\item[\textup{(i)}] For every $i$ with $2i<n$,
  \[
    \mathcal A^{i}(Y)=\Q\cdot L^{i},
    \qquad
    \mathcal A^{i}_{\mathrm{prim}}(Y)=
    \begin{cases}
      \Q & (i=0),\\[2pt]
      0  & (i\ge1),
    \end{cases}
  \]
  and Conjecture~\textup{\ref{conj:hdg}} holds for $Y$ in each such
  degree $i$. In particular, if $n$ is odd then
  Conjecture~\textup{\ref{conj:hdg}} holds for $Y$ outright.
\item[\textup{(ii)}] $H^{\mathrm{odd}}(Y,\Ql)=0$, and
  $H^{2k}(Y,\Ql(k))=\Ql\cdot L^{k}$ is one-dimensional for
  $0\le k\le n$, $k\ne m$.
\item[\textup{(iii)}] $\langle L^{m},L^{m}\rangle=\deg Y>0$.
\item[\textup{(iv)}] $\mathcal P^{n}(Y)=\ker\bigl(L:H^{n}(Y,\Ql(m))\to H^{n+2}(Y,\Ql(m+1))\bigr)
  =\bigl(L^{m}\bigr)^{\perp}$.
\item[\textup{(v)}] The pairing restricts to a nondegenerate form on each
  summand of the orthogonal decomposition
  \[
    H^{n}\bigl(Y,\Ql(m)\bigr)\;=\;\Ql\cdot L^{m}\;\perp\;\mathcal P^{n}(Y).
  \]
\end{enumerate}
Only weak Lefschetz and Poincar\'e duality are used\textup{;} the hard
Lefschetz theorem \textup{\cite[\S4.1]{DeligneWeilII}} is not required.
\end{proposition}

\begin{proof}
All statements follow from weak Lefschetz (Fact~\ref{fact:weakL}) and Poincar\'e
duality (Fact~\ref{fact:PD}). For $k\ne m$, $H^{2k}(Y,\Ql(k))=\Ql L^k$ with
$\deg L^n=\deg Y>0$. So for $2i<n$ no nonzero multiple of $L^i$ with $i\ge1$ is
primitive, and in degree $0$ the form is multiplication by $\deg Y$. For (iv),
$z\mapsto\deg(z\cupp L^{m-1})$ is an isomorphism
$H^{n+2}(Y,\Ql(m+1))\to\Ql$, so $L\cupp a=0$ if and only if $\langle a,L^m\rangle=0$.
Part (v) follows from (iii), (iv) and Fact~\ref{fact:PD}.
\end{proof}

\subsection{Results used in the proof}
\label{subsec:imports}

We list the external results we shall use, in the form in which we use
them.

\subsubsection*{(A) Weak Lefschetz, Poincar\'e duality, cycle classes}

\begin{fact}[Weak Lefschetz]
\label{fact:weakL}
If $Y\subset\mathbb P^{n}$ is a smooth hypersurface then the
restriction $H^{m}(\mathbb P^{n},\Ql)\to H^{m}(Y,\Ql)$ is an isomorphism for
$m<\dim Y$ and injective for $m=\dim Y$.
\end{fact}

This is the dual form of \cite[Ch.~VI, Thm.~7.1(a)]{Milne}, applied to the
$d$-uple Veronese embedding of $\mathbb P^n$.

\begin{fact}[Poincar\'e duality {\cite[Ch.~VI, Cor.~11.2]{Milne}}]
\label{fact:PD}
For $Y$ smooth projective of dimension $n$ the cup product
\[
  H^{m}(Y,\Ql(i))\times H^{2n-m}(Y,\Ql(n-i))\to H^{2n}(Y,\Ql(n))
  \xrightarrow{\ \deg\ }\Ql
\]
is a perfect pairing.
\end{fact}

\begin{fact}[Cycle classes {\cite[Ch.~VI, \S9]{Milne}},
{\cite[Ch.~19]{Fulton}}]
\label{fact:cycle}
There is a cycle class map $Z^{i}(Y)\to H^{2i}(Y,\Ql(i))$ compatible
with intersection products: the class of a refined intersection
product computed in the Chow group maps to the cup product of the
classes.
\end{fact}

\begin{fact}[Normal bundles {\cite[Ex.~II.8.4(a)]{Hartshorne}}]
\label{fact:normal}
For a smooth hypersurface $Y_d\subset\mathbb P^{n}$ of degree $d$ one
has $N_{Y_d/\mathbb P^{n}}\cong\mathcal O_{Y_d}(d)$; for a linear
subspace $\Pi$$=\mathbb P^{m}\subset\mathbb P^{n}$ one has
$N_{\Pi/\mathbb P^{n}}\cong\mathcal O_{\Pi}(1)^{\oplus(n-m)}$.
\end{fact}

\begin{fact}[Betti numbers of a nonsingular hypersurface]
\label{fact:chern}
Let $Y\subset\mathbb P^{n+1}$ be a nonsingular hypersurface of degree
$d$ and dimension $n$ over an algebraically closed field of
characteristic $\ne\ell$, and write $\beta_m(Y):=\dim_{\Ql}H^{m}(Y,\Ql)$. Then
$\beta_m(Y)=\beta_m(\mathbb P^n)$ for $m\ne n$, and
\begin{equation}
\label{eq:classicalbetti}
  \beta_n(Y)\;=\;\frac{(d-1)^{\,n+2}+(-1)^{n}(d-1)}{d}+\frac{1+(-1)^n}{2}
\end{equation}
\cite[Ch.~5, (3.3)]{Dimca}, \cite[Exp.~XI]{SGA7}; for Hermitian varieties see
also \cite{Shimada2001}.
\end{fact}

\subsubsection*{(B) Excess intersection}

Two closed subvarieties $P,P'$ of a smooth variety $Y$ meet
\emph{cleanly} along $Z=P\cap P'$ if $Z$ is smooth and
$T_zZ=T_zP\cap T_zP'$ inside $T_zY$ for every closed point $z\in Z$
\cite[\S6.3]{Fulton}. In that case the composite
$N_{Z/P}\hookrightarrow N_{Z/Y}\twoheadrightarrow N_{P'/Y}|_{Z}$ is an
injection of vector bundles and one defines the \emph{excess bundle}
\begin{equation}
\label{eq:excessbundle}
  E\;:=\;\bigl(N_{P'/Y}\big|_{Z}\bigr)\big/N_{Z/P},
\end{equation}
which is locally free of rank
$\dim Z-\bigl(\dim P+\dim P'-\dim Y\bigr)$.

\begin{fact}[Excess intersection formula {\cite[Thm.~6.3]{Fulton}}]
\label{fact:excess}
Let $Y$ be smooth, $P'\subset Y$ regularly embedded,
meeting $P$ cleanly along $Z$, with $Z\subset P$ regularly embedded.
If $\dim P+\dim P'=\dim Y$, then
\begin{equation}
\label{eq:excessformula}
  P\cdot P'\;=\;\deg\bigl(\chern_{e}(E)\capp[Z]\bigr),
  \qquad e\;=\;\mathrm{codim}(P',Y)-\mathrm{codim}(Z,P)\;=\;\mathrm{rank}\,E .
\end{equation}
The case $P=P'$ is the self-intersection formula
\cite[Cor.~6.3]{Fulton}, where $Z=P$, $N_{Z/P}=0$ and $E=N_{P/Y}$.
\end{fact}

\subsubsection*{(C) Witt's theorem and distance-regular graphs}

\begin{fact}[Witt's theorem for Hermitian forms {\cite[Thm.~7.4]{Taylor}}]
\label{fact:witt}
Any isometry between subspaces of a nondegenerate Hermitian space
extends to an isometry of the whole space. Consequently $\GU_{2r}(q)$
is transitive on pairs of maximal totally isotropic subspaces with a
given intersection dimension \textup{(}cf.\ the proof
of \cite[Thm.~9.4.3]{BCN}\textup{)}.
\end{fact}

Let $\Gamma$ be a connected graph. For vertices $u,v$ we write
$\partial(u,v)$ for the \emph{graph distance}, i.e.\ the length of a
shortest path from $u$ to $v$ (so $\partial(u,u)=0$ and
$\partial(u,v)=1$ means $u\sim v$); the \emph{diameter} is the
maximum of $\partial$. We use $\partial$ rather than the more common
$d(\cdot,\cdot)$ because $d$ is reserved throughout for $q+1=\deg X$.

Say that $\Gamma$, of diameter $r$, is \emph{distance-regular} if, for
all $0\le j\le r$, the three numbers
\[
\begin{aligned}
  c_j&=\#\bigl\{w:\ \partial(u,w)=j-1,\ \partial(v,w)=1\bigr\},\\
  a_j&=\#\bigl\{w:\ \partial(u,w)=j,\ \ \ \ \,\partial(v,w)=1\bigr\},\\
  b_j&=\#\bigl\{w:\ \partial(u,w)=j+1,\ \partial(v,w)=1\bigr\}
\end{aligned}
\]
depend only on $j$ and not on the choice of $u,v$ with
$\partial(u,v)=j$. For $0\le j\le r$ let $R_j$ be the distance-$j$
relation,
\[
  R_j\;=\;\bigl\{(u,v)\;:\;\partial(u,v)=j\bigr\},
\]
and let $A_j$ be its \emph{adjacency matrix}: the matrix indexed by the
vertex set of $\Gamma$ with $(A_j)_{u,v}=1$ if $\partial(u,v)=j$ and $0$
otherwise. Thus
$A_0=\mathrm{Id}$ is the identity matrix and $A:=A_1$ is the ordinary
adjacency matrix of $\Gamma$. 
Let $J$ be the square matrix, indexed by the vertex set of $\Gamma$, each of
whose entries is equal to $1$. Since every ordered pair $(u,v)$ lies
in exactly one relation $R_j$, we have
\[
  \sum_{j=0}^{r}A_j\;=\;J .
\]
Distance-regularity is then equivalent to the three-term recurrence
\begin{equation}
\label{eq:3termA}
  A\,A_j\;=\;b_{j-1}A_{j-1}+a_jA_j+c_{j+1}A_{j+1},
  \qquad 0\le j\le r,
\end{equation}
with the conventions $b_{-1}=c_{r+1}=0$, $c_0=0$, $b_r=0$; one has
$a_j=b_0-b_j-c_j$. The array $\{b_0,\dots,b_{r-1};c_1,\dots,c_r\}$ is
the \emph{intersection array}.

\begin{definition}[Valencies]
\label{def:valency}
Let $\Gamma$ be distance-regular of diameter $r$. For $0\le j\le r$
the \emph{$j$-th valency} of $\Gamma$ is
\[
  k_j\;:=\;\#\bigl\{w\in\Gamma\;:\;\partial(u,w)=j\bigr\},
\]
the number of vertices at distance $j$ from a fixed vertex $u$; it does
not depend on $u$. Thus $k_0=1$, $k_1$ is the valency, and
$A_jJ=JA_j=k_jJ$.
\end{definition}

\begin{lemma}[{\cite[\S4.1]{BCN}}]
\label{lem:valency}
$k_1=b_0$, and for $1\le j\le r$
\begin{equation}
\label{eq:valency}
  k_j\;=\;\frac{b_0b_1\cdots b_{j-1}}{c_1c_2\cdots c_j} .
\end{equation}
\end{lemma}

\begin{proof}
Count in two ways the pairs $(w,w')$ with $\partial(u,w)=j-1$,
$\partial(u,w')=j$ and $w\sim w'$. Counting from $w$ gives
$k_{j-1}b_{j-1}$, counting from $w'$ gives $k_jc_j$; hence
$k_j=k_{j-1}b_{j-1}/c_j$, and \eqref{eq:valency} follows by induction
from $k_0=1$.
\end{proof}

\begin{fact}[Dual polar graphs {\cite[\S9.4]{BCN}}]
\label{fact:dualpolar}
The graph whose vertices are the maximal totally isotropic subspaces of
a finite classical polar space, adjacent when they meet in a hyperplane
of each, is distance-regular of diameter equal to the rank, the
distance-$j$ relation being ``the intersection has codimension $j$''.
\end{fact}

\subsubsection*{(D) The $q$-binomial theorem}

For a prime power $Q$ and integers $0\le j\le s$, the
\emph{$Q$-binomial coefficient} is
\[
  \binom{s}{j}_{\!Q}:=\prod_{i=1}^{j}\frac{Q^{\,s-i+1}-1}{Q^{i}-1},
\]
the number of $j$-dimensional subspaces of $\F_{Q}^{\,s}$. It is used
here only through the following identity, which in turn is used only
in the proof of Lemma~\textup{\ref{lem:sumkj}}.

\begin{fact}[Gauss; {\cite[Thm.~3.3]{Andrews}}]
\label{fact:qbinom}
For an indeterminate $x$ and every $s\ge0$,
\[
  \prod_{i=0}^{s-1}\bigl(1+Q^{i}x\bigr)
  \;=\;\sum_{j=0}^{s}\binom{s}{j}_{\!Q}Q^{\binom j2}x^{j}.
\]
\end{fact}

\section{The Hermitian variety, its generators, and the two forms}
\label{sec:setup}

From here on everything is proved. The goal of this section is to put
in place the two bilinear forms that the rest of the argument
compares: the Poincar\'e pairing $\langle\,,\,\rangle$ on the middle
cohomology of $X$ (\S\ref{subsec:middle}), and its pullback to
$\Ql^{\mathcal G}$, whose Gram matrix $M$ is computed in
\S\ref{subsec:intnum} and identified in \S\ref{subsec:gram} as a
polynomial in a single symmetric matrix.

\subsection{Generators and the dual polar graph}
\label{subsec:generators}

We recall the Hermitian variety $\Herm(2r-1,q^2)$ and its generators,
the maximal $\F_{q^2}$-rational linear subspaces contained in it. The
graph on the generators is the dual polar graph ${}^2A_{2r-1}(q)$, and
we compute its intersection array, its valencies and two sums of
valencies needed later.

\begin{definition}
\label{def:herm}
The \emph{Hermitian variety} $\Herm(2r-1,q^2)$ is
the hypersurface
\[
 \sum_{i=0}^{2r-1}x_i^{\,q+1}=0
  \ \subset\ \mathbb P^{2r-1},
\]
defined over $\F_{q^2}$.
\end{definition}

Since $q+1\equiv1\pmod p$, the partial derivatives $(q+1)x_i^{\,q}$
vanish simultaneously only at the origin; hence $\Herm(2r-1,q^2)$ is
smooth over $\F_{q^2}$ in every characteristic $p$, of dimension $2r-2$. 

The form $\langle x,y\rangle_{\mathrm{H}}=\sum_i x_i\bar y_i$ with
$\bar x=x^{q}$ is a nondegenerate Hermitian form on $\F_{q^2}^{2r}$
relative to the involution $x\mapsto x^{q}$ of $\F_{q^2}/\F_q$, and
$\Herm(2r-1,q^2)$ is exactly its isotropic locus. Its automorphism
group, as a polarity-preserving subgroup of the ambient
$\mathbb P^{2r-1}$, is $\PGU_{2r}(q)$.

\begin{definition}
\label{def:generator}
A \emph{generator} of $\Herm(2r-1,q^2)$ is an $\F_{q^{2}}$-rational
linear subspace $P\subset\mathbb P^{2r-1}$ whose corresponding subspace
$\widetilde P$ of $\F_{q^2}^{2r}$
is totally isotropic for the Hermitian form, of the maximal possible
dimension, i.e., $\dim_{\F_{q^2}}\widetilde P=r$, so that
$\dim P=r-1$. Note $P\subset\Herm(2r-1,q^2)$
\textup{(}Lemma~\textup{\ref{lem:contained}}\textup{)}. Write $\mathcal G$ for the set of
generators.
\end{definition}

\begin{proposition}[classical; see {\cite[Def.~2.17, Lem.~2.18, Thm.~2.19]{HT}},
{\cite{Segre}}, {\cite{BoseChakravarti}}]
\label{prop:generators}
$\PGU_{2r}(q)$ acts transitively on $\mathcal G$, and
\[
  |\mathcal G|\;=\;N_r\;=\;\prod_{i=1}^{r}\bigl(q^{2i-1}+1\bigr).
\]
The graph on $\mathcal G$ with $P\sim P'$ iff $\dim(P\cap P')=r-2$ is
the dual polar graph of type ${}^{2}A_{2r-1}(q)$: it is
distance-regular of diameter $r$, its distance-$j$ relation being
\[
  R_j\;=\;\bigl\{(P,P')\;:\;\dim(P\cap P')=r-1-j\bigr\},
  \qquad 0\le j\le r .
\]
\end{proposition}

\begin{proof}
Transitivity is Fact~\ref{fact:witt}; the count is the classical one
\cite[Thm.~2.19]{HT}; distance-regularity is Fact~\ref{fact:dualpolar}.
\end{proof}

We write $A_j$ for the adjacency matrix of the relation $R_j$ in the
sense of \S\ref{subsec:imports}(C), and $A:=A_1$, so that
$A_0=\mathrm{Id}$ and $\sum_{j=0}^{r}A_j=J$, the matrix all of whose
entries equal $1$ (\S\ref{subsec:imports}(C)), and
the recurrence \eqref{eq:3termA} holds. Note that, for
$P,P'\in\mathcal G$, the graph distance is
$\partial(P,P')=r-1-\dim(P\cap P')$.

Lemma~\ref{lem:array} below is \cite[Thm.~9.4.3]{BCN}
\textup{(}case ${}^2A_{2d-1}(r)$ there, with $e=\frac12$, $d=r$, and
our $q^{2}$ in place of their $q$\textup{)}, and the value of $N_r$ in
the first identity of Lemma~\ref{lem:sumkj} is the case $k=r$ of
\cite[Lem.~9.4.1]{BCN}.

\begin{lemma}[Intersection array of ${}^2A_{2r-1}(q)$]
\label{lem:array}
For $0\le i\le r$,
\[
  c_i\;=\;\frac{q^{2i}-1}{q^{2}-1},\qquad
  b_i\;=\;q^{\,2i+1}\,\frac{q^{2(r-i)}-1}{q^{2}-1},
\]
with $c_0=0$ and $b_r=0$. In particular
$b_0=q\,\dfrac{q^{2r}-1}{q^{2}-1}$ is the valency. Moreover
\begin{equation}
\label{eq:aisc}
  a_i\;=\;(q-1)\,c_i \qquad (0\le i\le r).
\end{equation}
\end{lemma}

\begin{proof}
The values of $c_i,b_i$ are \cite[Thm.~9.4.3]{BCN}. For \eqref{eq:aisc},
substitute them into $a_i=b_0-b_i-c_i$.
\end{proof}

\begin{lemma}[Valencies of ${}^2A_{2r-1}(q)$]
\label{lem:kj}
Fix $P\in\mathcal G$ and $0\le j\le r$. The $j$-th valency of the dual
polar graph --- that is \textup{(}Definition~\ref{def:valency}\textup{)},
the number of $P'\in\mathcal G$ with $\partial(P,P')=j$, equivalently
with $\dim(P\cap P')=r-1-j$ --- equals
\[
  k_j\;=\;q^{\,j^{2}}\binom{r}{j}_{\!q^{2}} .
\]
For instance $k_0=1$ and $k_1=b_0=q\,\frac{q^{2r}-1}{q^{2}-1}$.
\end{lemma}

\begin{proof}
Substitute Lemma~\ref{lem:array} into \eqref{eq:valency}:
\[
  k_j=\frac{q^{\sum_{i=0}^{j-1}(2i+1)}
  \prod_{i=0}^{j-1}\frac{q^{2(r-i)}-1}{q^{2}-1}}
  {\prod_{i=1}^{j}\frac{q^{2i}-1}{q^{2}-1}}
  =q^{\,j^{2}}\prod_{i=1}^{j}\frac{q^{2(r-i+1)}-1}{q^{2i}-1}
  =q^{\,j^{2}}\binom{r}{j}_{\!q^{2}} . \qedhere
\]
\end{proof}

\begin{lemma}
\label{lem:sumkj}
For every $r\ge1$,
\[
  \sum_{j=0}^{r}k_j=N_r,
  \qquad
  \sum_{j=0}^{r}k_j\,q^{-2j}=\frac{q+1}{q}\,N_{r-1},
  \qquad
  \sum_{j=0}^{r}k_j\,(-q)^{\,r-j}=0 .
\]
\end{lemma}

\begin{proof}
All three are instances of the $q$-binomial theorem
(Fact~\ref{fact:qbinom}) with $Q=q^{2}$, so that
$Q^{\binom j2}=q^{\,j^{2}-j}$ and
$\binom rj_{q^2}Q^{\binom j2}x^{j}=k_j\,q^{-j}x^{j}$.

Taking $x=q$ gives $\sum_{j=0}^{r}k_j$ on the right and
$\prod_{i=0}^{r-1}(1+q^{2i+1})=\prod_{i=1}^{r}(1+q^{2i-1})=N_r$ on the
left.

Taking $x=q^{-1}$ gives $\sum_{j=0}^{r}k_j q^{-2j}$ on the right and
$\prod_{i=0}^{r-1}(1+q^{2i-1})=(1+q^{-1})\prod_{i=1}^{r-1}(1+q^{2i-1})
=\frac{q+1}{q}N_{r-1}$ on the left.

Taking $x=-1$ gives $\sum_{j=0}^{r}(-1)^{j}q^{\,j^{2}-j}\binom rj_{q^{2}}$ on
the right and $\prod_{i=0}^{r-1}(1-q^{2i})=0$ on the left, the factor
$i=0$ vanishing. Multiplying by $(-q)^{r}$ and using
$k_j=q^{\,j^{2}}\binom rj_{q^{2}}$ turns the right side into
$\sum_{j=0}^{r}k_j(-q)^{\,r-j}$.
\end{proof}

\subsection{The middle cohomology and the pairing $\langle\,,\,\rangle$}
\label{subsec:middle}

Write $X=\Herm(2r-1,q^2)$ from now on, $n=\dim X=2r-2$, and let $L\in
H^{2}(X,\Ql(1))$ be the hyperplane class. Define
\begin{equation}
\label{eq:pairingdef}
  \langle a,b\rangle\;:=\;\deg\bigl(a\cupp b\bigr),
  \qquad a,b\in H^{2r-2}\bigl(X,\Ql(r-1)\bigr),
\end{equation}
the Poincar\'e pairing on the middle cohomology; it is symmetric
(middle degree is even) and perfect (Fact~\ref{fact:PD}).

We recall the special case $Y=X$ of
Proposition~\ref{prop:belowmiddle}, in the form in which it will be
used.

\begin{proposition}[Lefschetz decomposition for $X$]
\label{prop:lefschetz}
Let $r\ge2$ and let $q$ be any prime power. Then\textup{:}
\begin{enumerate}
\item[\textup{(i)}] $H^{\mathrm{odd}}(X,\Ql)=0$, and the group
  $H^{2k}(X,\Ql(k))=\Ql\cdot L^{k}$ is one-dimensional for every
  $k$ with $0\le k\le 2r-2$ and $k\ne r-1$.
\item[\textup{(ii)}] $\langle L^{r-1},L^{r-1}\rangle=\deg(L^{2r-2})
  =\deg X=q+1>0$.
\item[\textup{(iii)}] $\mathcal P^{2r-2}(X)={}$$\ker\bigl(L:H^{2r-2}(X,\Ql(r-1))
  \to H^{2r}(X,\Ql(r))\bigr)=\bigl(L^{r-1}\bigr)^{\perp}$.
\item[\textup{(iv)}] The pairing restricts to a nondegenerate form on
  each summand of the orthogonal decomposition
  \[
    H^{2r-2}\bigl(X,\Ql(r-1)\bigr)\;=\;\Ql\cdot L^{r-1}
    \;\perp\;\mathcal P^{2r-2}(X),
  \]
  where
  \[
    \dim \mathcal P^{2r-2}(X)\;=\;\frac{q(q^{2r-1}+1)}{q+1}.
  \]
\end{enumerate}
\end{proposition}

\begin{proof}
This is Proposition~\ref{prop:belowmiddle} for $Y=X$, $n=2r-2$,
$m=r-1$, together with $\deg X=q+1$ \textup{(}items \textup{(i)--(iv)}
being Proposition~\ref{prop:belowmiddle}\textup{(ii)--(v)}\textup{)}. By
Fact~\ref{fact:chern} with $d=q+1$ and $n=2r-2$,
\begin{equation}
\label{eq:bettiherm}
  \beta_{2r-2}(X)\;=\;\frac{q^{2r}+q}{q+1}+1\;=\;\frac{q\,(q^{2r-1}+1)}{q+1}+1 ,
\end{equation}
and $\dim\mathcal P^{2r-2}(X)=\beta_{2r-2}(X)-1$.
\end{proof}

\subsection{Intersection numbers of generators}
\label{subsec:intnum}

The formula for the intersection numbers of two generators obtained in this
subsection (Lemma~\ref{lem:f}) is \cite[(4.1)]{Shimada2001}; we include
the derivation for completeness.

We first record, for later reference, the three immediate properties
of a generator that make its class available at all.

\begin{lemma}
\label{lem:contained}
Let $P\in\mathcal G$. Then \textup{(i)} $P\subset X$; \textup{(ii)}
$P$ is defined over $\F_{q^2}$; \textup{(iii)}
$\mathrm{codim}_XP=r-1$, so that
$[P]\in H^{2r-2}(X,\Ql(r-1))$ lies in middle cohomology.
\end{lemma}

\begin{proof}
Choose an $\F_{q^2}$-basis $e_1,\dots,e_r$ of $\widetilde P$. For
$v=\sum_i\lambda_ie_i\in\widehat P$ with $\lambda_i\in K$ we have
\[
  \sum_j v_j^{\,q+1}
  =\sum_{i,i'}\lambda_i\lambda_{i'}^{\,q}
   \langle e_i,e_{i'}\rangle_{\mathrm H}=0 ,
\]
since $\widetilde P$ is totally isotropic; as $\sum_jx_j^{\,q+1}$ is the
equation of $X$, this gives $P\subset X$. Next, $\widetilde P$ is an
$\F_{q^2}$-subspace by Definition~\ref{def:generator}, so $P$ is
defined over $\F_{q^2}$; and $\dim P=r-1$, $\dim X=2r-2$.
\end{proof}

Fix now $P,P'\in\mathcal G$ and set
\[
  Z:=P\cap P',\qquad j:=\dim Z\in\{-1,0,1,\dots,r-1\},
\]
with $j=-1$ meaning $Z=\emptyset$; for $j\ge0$, $Z\cong\mathbb P^{j}$.
Write $d=q+1=\deg X$ and $h=\chern_1(\mathcal O(1))$; for any
subvariety $V\subset\mathbb P^{2r-1}$ we use the same letter $h$ for
the restriction of $\chern_1(\mathcal O(1))$ to $V$. Both $[P]$ and $[P']$
lie in $H^{2(r-1)}(X,\Ql(r-1))$ and $\dim X=2(r-1)$, so we may form
\[
  P\cdot P'\;:=\;\deg\bigl([P]\cupp[P']\bigr)
  \;=\;\bigl\langle[P],[P']\bigr\rangle .
\]
The \emph{expected} dimension of $P\cap P'$ inside $X$ is
$\dim P+\dim P'-\dim X=0$, while the actual dimension is $j$; so the
intersection carries excess $j$ and $P\cdot P'$ is not a naive point
count. It is computed by Fact~\ref{fact:excess}. We carry this out in
detail, since the resulting numbers are the sole geometric input to
everything that follows.

\begin{lemma}[Cleanness]
\label{lem:clean}
Suppose $j\ge0$. Then $Z$ is smooth and $T_zZ=T_zP\cap T_zP'$ inside
$T_zX$ for every closed point $z\in Z$; that is, $P$ and $P'$ meet
cleanly along $Z$.
\end{lemma}

\begin{proof}
$Z\cong\mathbb P^{j}$ is smooth. Let $\widehat P,\widehat P'\subset
K^{2r}$ be the $r$-dimensional linear subspaces with
$P=\mathbb P(\widehat P)$, $P'=\mathbb P(\widehat P')$, and let
$\ell_z\subset K^{2r}$ be the line corresponding to $z$. The standard
identification of the tangent space of a projective space with a
$\mathrm{Hom}$-space \cite[Thm.~II.8.13]{Hartshorne} gives
\[
  T_z\mathbb P^{2r-1}=\mathrm{Hom}\bigl(\ell_z,K^{2r}/\ell_z\bigr),
  \qquad
  T_zP=\mathrm{Hom}\bigl(\ell_z,\widehat P/\ell_z\bigr),
\]
and likewise for $P'$ and for
$Z=\mathbb P(\widehat P\cap\widehat P')$. Since
$\ell_z\subset\widehat P\cap\widehat P'$, the images of
$\widehat P/\ell_z$ and $\widehat P'/\ell_z$ in $K^{2r}/\ell_z$ meet
exactly in $(\widehat P\cap\widehat P')/\ell_z$; as
$\mathrm{Hom}(\ell_z,-)$ is left exact it preserves this intersection,
whence $T_zP\cap T_zP'=T_zZ$. Finally $T_zP,T_zP'\subset T_zX\subset
T_z\mathbb P^{2r-1}$ because $P,P'\subset X$ by
Lemma~\ref{lem:contained}(i), so the intersection computed
inside $T_zX$ is the same.
\end{proof}

\begin{lemma}[Chern class of the excess bundle]
\label{lem:excessbundle}
Let $j\ge0$ and let $E$ be the excess bundle \eqref{eq:excessbundle}
for $P,P'\subset X$. Then $E$ is a vector bundle on
$Z\cong\mathbb P^{j}$ of rank $j$, with total Chern class
\[
  \chern(E)\;=\;\frac{(1+h)^{\,j+1}}{1+d\,h},\qquad d=q+1 .
\]
\end{lemma}

\begin{proof}
By Lemma~\ref{lem:clean} the bundle $E$ is defined. We compute the two
bundles in \eqref{eq:excessbundle}.

\textup{(a)} By Fact~\ref{fact:normal} applied to $P'=\mathbb P^{r-1}
\subset\mathbb P^{2r-1}$, $N_{P'/\mathbb P^{2r-1}}\cong
\mathcal O_{P'}(1)^{\oplus r}$; note the codimension of a generator in
$\mathbb P^{2r-1}$ is $r$, not $r-1$.

\textup{(b)} $N_{X/\mathbb P^{2r-1}}\cong\mathcal O_X(d)$, $X$ being a
smooth hypersurface of degree $d$ (Fact~\ref{fact:normal}).

\textup{(c)} From $P'\subset X\subset\mathbb P^{2r-1}$ we get
\[
  0\to N_{P'/X}\to N_{P'/\mathbb P^{2r-1}}
  \to N_{X/\mathbb P^{2r-1}}\big|_{P'}\to0,
\]
so by (a), (b) and Whitney,
\[
  \chern\bigl(N_{P'/X}\bigr)=\frac{(1+h)^{r}}{1+d\,h},
  \qquad\mathrm{rank}\,N_{P'/X}=r-1 .
\]

\textup{(d)} By Fact~\ref{fact:normal} applied to
$Z=\mathbb P^{j}\subset P=\mathbb P^{r-1}$,
$N_{Z/P}\cong\mathcal O_Z(1)^{\oplus(r-1-j)}$.

Combining, $E$ has rank $(r-1)-(r-1-j)=j$ and
\[
  \chern(E)=\frac{\chern\bigl(N_{P'/X}|_{Z}\bigr)}{\chern\bigl(N_{Z/P}\bigr)}
  =\frac{(1+h)^{r}/(1+d\,h)}{(1+h)^{\,r-1-j}}
  =\frac{(1+h)^{\,j+1}}{1+d\,h} . \qedhere
\]
\end{proof}

\begin{lemma}[Intersection numbers]
\label{lem:f}
Let $P,P'\in\mathcal G$ and $j=\dim(P\cap P')$. Then $P\cdot P'$
depends only on $j$, and
\begin{equation}
\label{eq:fclosed}
  P\cdot P'\;=\;f(j)\;:=\;\frac{1-(-q)^{\,j+1}}{q+1},\qquad j\ge-1 .
\end{equation}
In particular $f(-1)=0$, and $f(r-1)=P\cdot P$ is the common
self-intersection number of a generator:
\[
  f(0)=1,\quad f(1)=1-q,\quad f(2)=q^{2}-q+1,\quad
  f(3)=-(q-1)(q^{2}+1),\ \dots
\]
\end{lemma}

\begin{proof}
If $j=-1$ then $P\cap P'=\emptyset$ and $P\cdot P'=0$, agreeing with
\eqref{eq:fclosed}. Assume $j\ge0$. By Lemma~\ref{lem:clean} the
hypotheses of Fact~\ref{fact:excess} hold, with
$\mathrm{codim}(P',X)=r-1$ and $\mathrm{codim}(Z,P)=r-1-j$, so
$e=j=\mathrm{rank}\,E$ in agreement with
Lemma~\ref{lem:excessbundle}. Here $Z\cong\mathbb P^{j}$ and $h$ is
the restriction of $\chern_1(\mathcal O(1))$ to $Z$, so that
$A^{*}(Z)=\Z[h]/(h^{j+1})$. In this ring
Lemma~\ref{lem:excessbundle} gives
$\chern(E)=(1+h)^{j+1}(1+dh)^{-1}$ (a finite sum, $h$ being
nilpotent), with $d=q+1$. Expanding $(1+h)^{j+1}$ and
$(1+dh)^{-1}=\sum_{k\ge0}(-d)^{k}h^{k}$ and collecting the terms of
degree $j$ gives
\[
  \chern_j(E)\;=\;\Bigl(\sum_{i=0}^{j}\binom{j+1}{i}(-d)^{\,j-i}\Bigr)h^{j}.
\]
Since $\deg\bigl(h^{j}\capp[Z]\bigr)=1$, \eqref{eq:excessformula} gives
\[
  P\cdot P'\;=\;\deg\bigl(\chern_j(E)\capp[Z]\bigr)
  \;=\;\sum_{i=0}^{j}\binom{j+1}{i}(-d)^{\,j-i}.
\]
(That this rational-equivalence computation gives the cup
product is Fact~\ref{fact:cycle}.)

It remains to evaluate this sum. Apply the binomial theorem in the form
$\sum_{i=0}^{j+1}\binom{j+1}{i}(-d)^{\,j+1-i}=(1-d)^{\,j+1}$ and
subtract the term $i=j+1$, which is $1$:
\[
  \sum_{i=0}^{j}\binom{j+1}{i}(-d)^{\,j+1-i}=(1-d)^{\,j+1}-1 .
\]
Dividing by $-d$ and using $1-d=-q$ gives
$\sum_{i=0}^{j}\binom{j+1}{i}(-d)^{\,j-i}=\frac{(1-d)^{j+1}-1}{-d}=\frac{1-(-q)^{\,j+1}}{q+1}$.
\end{proof}

\subsection{$\Ql^{\mathcal G}$, the cycle map, and the Gram matrix}
\label{subsec:gram}

We now pass from cycles to linear algebra: the cycle map sends the
basis vector $e_P$ of $\Ql^{\mathcal G}$ to $[P]$, and the intersection
pairing pulls back to a form whose Gram matrix $M$ is given by
Lemma~\ref{lem:f}. We express $M$ in terms of the distance matrices of
the dual polar graph and describe the resulting eigenspace
decomposition of $\Ql^{\mathcal G}$.

\begin{definition}
\label{def:phi}
Let $\Ql^{\mathcal G}$ be the $\Ql$-vector space with basis
$\{e_P\}_{P\in\mathcal G}$, equipped with the standard inner product
$\langle\,,\,\rangle_{\mathrm{std}}$ making $\{e_P\}$ orthonormal. Let
\[
  \phi\;:\;\Ql^{\mathcal G}\;\longrightarrow\;
  H^{2r-2}\bigl(X_{\overline\F_p},\Ql(r-1)\bigr),
  \qquad e_P\longmapsto[P],
\]
(well defined by Lemma~\ref{lem:contained}), let
$V:=\mathrm{image}(\phi)=\mathrm{span}_{\Ql}\{[P]:P\in\mathcal G\}$,
and let
\[
  M\;:=\;\bigl(P\cdot P'\bigr)_{P,P'\in\mathcal G}
  \;=\;\Bigl(\bigl\langle\phi(e_P),\phi(e_{P'})\bigr\rangle\Bigr)_{P,P'}
\]
be the Gram matrix of the pullback $\phi^{*}\langle\,,\,\rangle$.
\end{definition}

Thus $M$ is an $N_r\times N_r$ symmetric \emph{integer} matrix
(Lemma~\ref{lem:f}), and the entire remaining problem is to compare
$M$ with the standard inner product.

\begin{proposition}
\label{prop:gram}
With the notation of \S\ref{subsec:generators}\textup{:}
\begin{enumerate}
\item[\textup{(i)}] $M=\sum_{j=0}^{r}f(r-1-j)\,A_j$.
\item[\textup{(ii)}] $A_j=p_j(A)$ for a polynomial $p_j\in\Q[x]$ of
  degree exactly $j$; consequently
  \[
    M\;=\;F(A),\qquad F:=\sum_{j=0}^{r}f(r-1-j)\,p_j\in\Q[x],
    \quad\deg F\le r .
  \]
\item[\textup{(iii)}] $A$ is a symmetric integer matrix with exactly $r+1$
  distinct real eigenvalues $\theta_0>\theta_1>\cdots>\theta_r$, and
  $A\mathbf 1=b_0\mathbf 1$ for the
vector $\mathbf 1=\sum_P e_P$ all of whose coordinates equal $1$. Put
  $W_0:=\Ql\cdot\mathbf 1$ and $m_0:=1$.
\item[\textup{(iv)}] If $\theta\in\Q$ is an eigenvalue of $A$, every matrix of the form $\sum_j \lambda_jA_j$ acts on
  $\ker(A-\theta)\subset\Ql^{\mathcal G}$ by a scalar; in particular $M$ acts on $W_0$ by
  $\mu_0:=F(b_0)$.
\end{enumerate}
\end{proposition}

\begin{proof}
\textup{(i)} By Lemma~\ref{lem:f}, $P\cdot P'$ depends only on
$\dim(P\cap P')$; the entry of $M$ at $(P,P')\in R_j$ is the constant
$f(r-1-j)$.

\textup{(ii)} Induction on $j$: $A_0=\mathrm{Id}=p_0(A)$ and
$A_1=A=p_1(A)$; solving \eqref{eq:3termA} for $A_{j+1}$,
\[
  A_{j+1}=\frac{1}{c_{j+1}}\bigl(A\,A_j-a_jA_j-b_{j-1}A_{j-1}\bigr)
  =\frac{1}{c_{j+1}}\bigl(x\,p_j-a_jp_j-b_{j-1}p_{j-1}\bigr)(A),
\]
legitimate since $c_{j+1}\ne0$ for $j<r$ (Lemma~\ref{lem:array}); the
degree rises by exactly one at each step, the leading coefficient
being $1/(c_1\cdots c_j)\ne0$.

\textup{(iii)} $A$ is real symmetric, hence orthogonally
diagonalisable. The $A_j$ are nonzero with pairwise disjoint supports,
hence linearly independent; by (ii) they span the same space as
$\mathrm{Id},A,\dots,A^{r}$, so those $r+1$ powers are linearly
independent and the minimal polynomial of $A$ has degree exactly
$r+1$. A symmetric matrix has minimal polynomial
$\prod_\theta(x-\theta)$ over its distinct eigenvalues, so $A$ has
exactly $r+1$ of them. The graph is connected of diameter $r$ and
regular of valency $b_0$, so $A\mathbf 1=b_0\mathbf 1$.

\textup{(iv)} Any $\sum_j\lambda_jA_j=\bigl(\sum_j\lambda_jp_j\bigr)(A)$ is a
polynomial in $A$, hence acts on $\ker(A-\theta)$ by the scalar
obtained by substituting $\theta$.
\end{proof}

\begin{remark}
\label{rem:allW}
By \cite[Thm.~9.4.3]{BCN}, the eigenvalues of $A$ are
$\theta_j=q\frac{q^{2(r-j)}-1}{q^{2}-1}-\frac{q^{2j}-1}{q^{2}-1}$ $(0\le j\le r)$. They are
integers, $\theta_0=b_0$ is simple and $\theta_r=-c_r$. Hence each
$W_j:=\ker(A-\theta_j)\subset\Ql^{\mathcal G}$ is defined over $\Q$, and the spectral theorem
gives $\Ql^{\mathcal G}=W_0\perp W_1\perp\cdots\perp W_r$ with $W_0=\Ql\mathbf 1$. We
shall not use this: only $W_0$ and $W_r$ enter, and what we need of them is proved
directly (Theorem~\ref{thm:bottomrow}, Lemmas~\ref{rem:thetastarmin}
and~\ref{lem:Wprime}). In the notation of Lemma~\ref{lem:Wprime},
$W'=W_1\perp\cdots\perp W_{r-1}$.
\end{remark}

Two further elementary facts will be needed. The first is the only
place where the hyperplane class meets the generators; the second
identifies the eigenvalue $\mu_0$ before any spectral analysis.

\begin{lemma}[Degree of a generator, and row sums]
\label{lem:eig0}
For every $P\in\mathcal G$\textup{:}
\begin{enumerate}
\item[\textup{(i)}] $\bigl\langle[P],L^{r-1}\bigr\rangle=1$;
\item[\textup{(ii)}] $\displaystyle\sum_{P'\in\mathcal G}P\cdot P'
  \;=\;\sum_{j=0}^{r}k_j\,f(r-1-j)\;=\;\frac{N_r}{q+1}$, i.e.\
  $M\mathbf 1=\frac{N_r}{q+1}\mathbf 1$ and
  $\mu_0=\frac{N_r}{q+1}$.
\end{enumerate}
\end{lemma}

\begin{proof}
(i) $P\cong\mathbb P^{r-1}$ meets a general linear subspace of
$\mathbb P^{2r-1}$ of codimension $r-1$ transversally in a
single reduced point.

(ii) The first equality holds because, by Lemma~\ref{lem:f}, the entry
$P\cdot P'$ equals $f(r-1-j)$ whenever $\dim(P\cap P')=r-1-j$, and by
Lemma~\ref{lem:kj} there are $k_j$ such $P'$. For the second, insert
the closed form \eqref{eq:fclosed}, $f(r-1-j)=\frac{1-(-q)^{\,r-j}}{q+1}$,
and apply Lemma~\ref{lem:sumkj} twice:
\[
  \sum_{j=0}^{r}k_j\,f(r-1-j)
  =\frac{1}{q+1}\Bigl(\sum_{j}k_j-\sum_{j}k_j(-q)^{\,r-j}\Bigr)
  =\frac{N_r-0}{q+1}=\frac{N_r}{q+1}.
\]
The row sums of $M$ are thus all equal to $\frac{N_r}{q+1}$, which
says exactly $M\mathbf 1=\frac{N_r}{q+1}\mathbf 1$; and $\mathbf 1$
spans $W_0$ by definition.
\end{proof}

\section{The orthogonal decomposition of $\Ql^{\mathcal G}$}
\label{sec:decomp}

We now determine the spectrum of $M$ completely, and with it the
decomposition of $\Ql^{\mathcal G}$ relative to $\phi$. Everything in
\S\S\ref{subsec:reform}--\ref{subsec:spectrum} is pure linear algebra
and combinatorics, using only Lemma~\ref{lem:array},
Lemma~\ref{lem:kj}, Lemma~\ref{lem:sumkj}, Lemma~\ref{lem:f} and
Proposition~\ref{prop:gram}.

\subsection{A cleaner form of $M$}
\label{subsec:reform}

Since $\sum_{j=0}^{r}A_j=J$, the closed form \eqref{eq:fclosed} for
$f$ gives at once
\begin{equation}
\label{eq:Mreform}
  M\;=\;\sum_{j=0}^{r}f(r-1-j)A_j
  \;=\;\frac{1}{q+1}\bigl(J-M'\bigr),
  \qquad
  M'\;:=\;\sum_{j=0}^{r}(-q)^{\,r-j}A_j .
\end{equation}
Concretely $M'$ has the transparent description
\begin{equation}
\label{eq:Mprime-entries}
  M'_{P,P'}\;=\;(-q)^{\dim_{\F_{q^2}}(\widetilde P\cap\widetilde P')} ,
\end{equation}
since a pair at distance $j$ has $\dim_{\F_{q^2}}(\widetilde P\cap
\widetilde P')=r-j$.

\medskip

\noindent
The matrix $J=\mathbf 1\mathbf 1^{\mathsf T}$ is supported on $W_0$: it
acts by $N_r$ on $W_0=\Ql\mathbf 1$ and by $0$ on the orthogonal complement $\mathbf 1^{\perp}$. So the whole
spectrum of $M$ is known as soon as we know that of $M'$, and what we
shall prove is that $M'$ too is supported on a \emph{single}
eigenspace of $A$.

\subsection{The key identity}
\label{subsec:key}

We show that $M'$ is supported on a single eigenspace of $A$, by
checking directly from the intersection array that $AM'$ is a scalar
multiple of $M'$. We then identify this eigenspace as the one
belonging to the smallest eigenvalue of $A$, and compute the scalar by
which $M'$ acts on it and its dimension.

\begin{theorem}[Key identity]
\label{thm:bottomrow}
Put
\[
  \theta_\star\;:=\;-c_r\;=\;-\,\frac{q^{2r}-1}{q^{2}-1}.
\]
Then, for 
every prime power $q$,
\begin{equation}
\label{eq:bottomrow}
  A\,M'\;=\;\theta_\star\,M' .
\end{equation}
Consequently $\theta_\star$ is an eigenvalue of $A$; writing
$W_\star:=\ker(A-\theta_\star)$, we have
$\mathrm{im}\,M'\subseteq W_\star$, and $W_\star\ne W_0$.
\end{theorem}

\begin{proof}
Write $g_j:=(-q)^{\,r-j}$, so $M'=\sum_{j=0}^{r}g_jA_j$. Applying
\eqref{eq:3termA} and collecting the coefficient of $A_m$,
\[
  A\,M'\;=\;\sum_{m=0}^{r}
  \bigl(g_{m+1}b_m+g_ma_m+g_{m-1}c_m\bigr)A_m ,
\]
and $g_{m+1}=-g_m/q$, $g_{m-1}=-q\,g_m$ turn the bracket into
$g_m\bigl[-\tfrac1q b_m+a_m-q\,c_m\bigr]$. So \eqref{eq:bottomrow}
amounts to the assertion that
\begin{equation}
\label{eq:bracket}
  -\tfrac1q b_m+a_m-q\,c_m\;=\;\theta_\star
  \qquad (0\le m\le r).
\end{equation}
The point is that the left-hand side does not depend on $m$. But
\eqref{eq:bracket} is immediate from the explicit formulas of
Lemma~\ref{lem:array}: by \eqref{eq:aisc} the left-hand side equals
$-\tfrac1q b_m-c_m$, in which the terms $q^{2m}$ cancel, leaving
$-c_r=\theta_\star$. The conventions $b_{-1}=c_{r+1}=0$,
$c_0=0$, $b_r=0$ make the computation valid at $m=0$ and $m=r$ as
well.

Finally $M'\ne0$, so \eqref{eq:bottomrow} exhibits any nonzero column
of $M'$ as a $\theta_\star$-eigenvector of $A$; every vector in
$\mathrm{im}\,M'$ is a combination of columns, whence
$\mathrm{im}\,M'\subseteq W_\star$. And $W_\star\ne W_0$ because
$\theta_\star<0<b_0$ and $A\mathbf 1=b_0\mathbf 1$.
\end{proof}

\begin{lemma}[$\theta_\star=\theta_r$ and $W_\star=W_r$]
\label{rem:thetastarmin}
The eigenvalue $\theta_\star$ is $\theta_r$, the smallest of the
$r+1$ eigenvalues $\theta_0>\theta_1>\cdots>\theta_r$ of
Proposition~\ref{prop:gram}(iii), and consequently $W_\star=W_r$.
\end{lemma}

\begin{proof}
For a totally isotropic $\F_{q^{2}}$-rational $Z\subset\mathbb P^{2r-1}$ with $\dim Z=r-2$,
the quotient $\widetilde Z^{\perp}/\widetilde Z$ carries a nondegenerate
rank-$2$ Hermitian form, so the generators $P\in\mathcal G$ with $P\supset Z$ --- the
\emph{pencil} through $Z$ --- number $N_1=q+1$, any two of them
meeting exactly in $Z$. Every generator $P$ lies on exactly
$\binom{r}{1}_{q^{2}}=c_r$ pencils, one for each hyperplane of
$P\cong\mathbb P^{r-1}$, and distinct
pencils meet in at most one generator, since if $P\ne P'$ both
contained two distinct such $Z_1,Z_2$, then $\widetilde P\cap\widetilde P'$
would contain $\widetilde Z_1+\widetilde Z_2$, which has dimension $r$,
forcing $P=P'$. Consequently, if $C$ is the generator--pencil
incidence matrix, that is, the $\{0,1\}$-matrix whose
rows are indexed by $\mathcal G$ and whose columns are indexed by the
totally isotropic $\F_{q^{2}}$-rational $Z\subset\mathbb P^{2r-1}$ with $\dim Z=r-2$, with
$C_{P,Z}=1$ if $Z\subset P$ and $C_{P,Z}=0$ otherwise, so that
$(CC^{\mathsf T})_{P,P'}$ is the number of such $Z$ contained in
$P\cap P'$, then $CC^{\mathsf T}=A+c_rI$; being a Gram matrix,
$CC^{\mathsf T}\succeq0$, so every eigenvalue of $A$ is $\ge-c_r$.
Since $\theta_\star=-c_r$ is itself an eigenvalue
(Theorem~\ref{thm:bottomrow}), it must be the least one.
\end{proof}

Since Lemma~\ref{rem:thetastarmin} identifies the index $\star$ as $r$, we write
$\theta_r$ and $W_r$ for $\theta_\star$ and $W_\star$ from now on.

\begin{lemma}[Orthogonal decomposition]
\label{lem:Wprime}
Put $W':=(W_0\oplus\Wst)^{\perp}\subset\Ql^{\mathcal G}$, the orthogonal complement for
$\langle\,,\,\rangle_{\mathrm{std}}$. Then
\[
  \Ql^{\mathcal G}\;=\;W_0\perp\Wst\perp W' .
\]
\end{lemma}

\begin{proof}
For $x\in W_0$ and $y\in\Wst$ we have
$b_0\,x^{\mathsf T}y=(Ax)^{\mathsf T}y=x^{\mathsf T}Ay=\theta_r\,x^{\mathsf T}y$, and
$b_0>0>\theta_r$; hence $W_0\perp\Wst$. Both spaces are of the form
$U\otimes_\Q\Ql$ with $U\subset\Q^{\mathcal G}$: for $W_0$ take $U=\Q\mathbf 1$, and
for $\Wst=\ker(A-\theta_r)$ take $U=\ker_\Q(A-\theta_r)$, since $A-\theta_r$ is an
integer matrix ($\theta_r=-c_r\in\Z$) and a basis of its kernel found by Gaussian
elimination over $\Q$ remains a basis over $\Ql$. The standard inner product is
positive definite on $\Q^{\mathcal G}$, so its Gram matrix in a $\Q$-basis of
$\Q\mathbf 1\oplus\ker_\Q(A-\theta_r)$ is invertible, and it stays invertible over
$\Ql$. Thus $\langle\,,\,\rangle_{\mathrm{std}}$ is nondegenerate on $W_0\oplus\Wst$, and
$\Ql^{\mathcal G}=(W_0\oplus\Wst)\oplus W'$.
\end{proof}

\begin{corollary}
\label{cor:Mprime}
$\mathrm{im}\,M'=\Wst$, and $M'$ acts on $\Wst$ by a nonzero scalar
$\eta_{r}$ and annihilates $W_0\perp W'$. Explicitly
\[
  \eta_{r}\;=\;(-1)^{r}q^{\,r-1}(q+1)N_{r-1},
  \qquad
  m_{r}\;:=\;\dim\Wst\;=\;\frac{q\,(q^{2r-1}+1)}{q+1}.
\]
\end{corollary}

\begin{proof}
By Proposition~\ref{prop:gram}(iv), applied with $\theta=\theta_r\in\Z$, $M'$
acts on $\Wst=\ker(A-\theta_r)$ by a scalar $\eta_r$. By
Theorem~\ref{thm:bottomrow}, $\mathrm{im}\,M'\subseteq\Wst$; since
$M'$ is symmetric,
$\ker M'=(\mathrm{im}\,M')^{\perp}\supseteq\Wst^{\perp}
=W_0\perp W'$ (Lemma~\ref{lem:Wprime}). Hence
$M'$ is $\eta_{r}$ times the orthogonal projection onto $\Wst$;
equivalently
\begin{equation}
\label{eq:Msqr}
  M'^{\,2}\;=\;\eta_{r}\,M' .
\end{equation}

To evaluate $\eta_{r}$, compare a single diagonal entry of
\eqref{eq:Msqr}. Fix $P\in\mathcal G$. Then $M'_{P,P}=(-q)^{r}$, while
by \eqref{eq:Mprime-entries} and Lemma~\ref{lem:kj}
\[
  \bigl(M'^{\,2}\bigr)_{P,P}=\sum_{P'\in\mathcal G}\bigl(M'_{P,P'}\bigr)^{2}
  =\sum_{j=0}^{r}k_j\,q^{\,2(r-j)}
  =q^{2r}\sum_{j=0}^{r}k_j\,q^{-2j}
  =q^{2r}\cdot\frac{q+1}{q}N_{r-1},
\]
the last step by Lemma~\ref{lem:sumkj}. Therefore
\[
  \eta_{r}=\frac{\bigl(M'^{\,2}\bigr)_{P,P}}{M'_{P,P}}
  =\frac{q^{2r}\cdot\frac{q+1}{q}N_{r-1}}{(-q)^{r}}
  =(-1)^{r}q^{\,r-1}(q+1)N_{r-1}\;\ne\;0 .
\]

For the multiplicity, take traces: $M'$ is $\eta_{r}$ times a
projection of rank $\dim\Wst$, so $\mathrm{tr}\,M'=m_{r}\eta_{r}$;
on the other hand only $A_0=\mathrm{Id}$ contributes to the diagonal,
so $\mathrm{tr}\,M'=(-q)^{r}N_r$. Using $N_r=N_{r-1}(q^{2r-1}+1)$,
\[
  m_{r}=\frac{(-q)^{r}N_r}{\eta_{r}}
  =\frac{(-q)^{r}N_{r-1}(q^{2r-1}+1)}{(-1)^{r}q^{\,r-1}(q+1)N_{r-1}}
  =\frac{q\,(q^{2r-1}+1)}{q+1} . \qedhere
\]
\end{proof}

\subsection{The spectrum of $M$}
\label{subsec:spectrum}

Combining $M=\frac{1}{q+1}(J-M')$ with the results of
\S\ref{subsec:key}, we obtain the complete spectrum of $M$: it has
exactly two nonzero eigenvalues.

\begin{theorem}[Spectrum of $M$]
\label{thm:spectrum}
For every $r\ge2$ and every prime power $q$, the symmetric matrix $M$
has exactly two nonzero eigenvalues:
\[
\begin{aligned}
  &\mu_0=\frac{N_r}{q+1}
  &&\text{on }W_0,
  &&m_0=1,\\[2pt]
  &\mu_{r}=(-1)^{r-1}q^{\,r-1}N_{r-1}
  &&\text{on }\Wst,
  &&m_{r}=\frac{q(q^{2r-1}+1)}{q+1},
\end{aligned}
\]
and $M$ annihilates $W'$. In
particular
\[
  \mathrm{rank}\,M\;=\;1+m_{r}\;=\;\frac{q(q^{2r-1}+1)}{q+1}+1,
  \qquad
  \ker M\;=\;W' .
\]
\end{theorem}

\begin{proof}
By \eqref{eq:Mreform}, $M=\frac{1}{q+1}(J-M')$. The matrix $J$ is
supported on $W_0$ with eigenvalue $N_r$ there, and by
Corollary~\ref{cor:Mprime} the matrix $M'$ is supported on $\Wst$ with
eigenvalue $\eta_{r}$ there; both vanish on $W'$, and
$\Ql^{\mathcal G}=W_0\perp\Wst\perp W'$ by Lemma~\ref{lem:Wprime}. Hence
\[
  \mu_0=\frac{N_r}{q+1},
  \qquad
  \mu_{r}=-\frac{\eta_{r}}{q+1}
  =-\frac{(-1)^{r}q^{\,r-1}(q+1)N_{r-1}}{q+1}
  =(-1)^{r-1}q^{\,r-1}N_{r-1}.
\]
Both are nonzero, so the rank and kernel are as stated, using
$m_{r}$ from Corollary~\ref{cor:Mprime}.
\end{proof}

\subsection{The generators span, and the decomposition}
\label{subsec:span}

Comparing $\mathrm{rank}\,M$ \textup{(}Theorem~\ref{thm:spectrum}\textup{)}
with the Betti number \eqref{eq:bettiherm}, we show that the classes $[P]$
span the middle cohomology, and we identify $W_0$ and $\Wst$ with the
non-primitive line and the primitive part.

\begin{theorem}[Decomposition of $\Ql^{\mathcal G}$]
\label{thm:decomp}
Let $r\ge2$ and let $q$ be any prime power. Then\textup{:}
\begin{enumerate}
\item[\textup{(i)}] The classes $\{[P]:P\in\mathcal G\}$ span the
  middle cohomology:
  \[
    V\;=\;H^{2r-2}\bigl(X_{\overline\F_p},\Ql(r-1)\bigr),
    \qquad \dim_{\Ql}V=\beta_{2r-2}(X)=1+m_{r} .
  \]
\item[\textup{(ii)}] $\ker\phi=\ker M=W'$,
  and this is exactly the radical of $\phi^{*}\langle\,,\,\rangle$. In
  particular there is an orthogonal decomposition
  \[
    \Ql^{\mathcal G}\;=\;W_0\ \perp\ \Wst\ \perp\ W' .
  \]
\item[\textup{(iii)}] $\phi$ maps $W_0$ isomorphically onto the
  non-primitive line $\Ql L^{r-1}$, and $\Wst$ isomorphically onto the
  primitive part $\mathcal P^{2r-2}(X)$.
\item[\textup{(iv)}] Under $\phi|_{\Wst}:\Wst\xrightarrow{\sim}
  \mathcal P^{2r-2}(X)$, the Poincar\'e pairing pulls back to $\mu_{r}$ times
  the standard inner product on $\Wst$.
\item[\textup{(v)}] Writing $\mathrm N^{r-1}(X)$ for the group of
  codimension-$(r-1)$ cycles modulo numerical equivalence,
  \[
    \mathrm N^{r-1}(X)\otimes\Ql\;\xrightarrow{\ \sim\ }\;
    H^{2r-2}\bigl(X_{\overline\F_p},\Ql(r-1)\bigr),
  \]
  and this isomorphism carries a $\Q$-structure, spanned by the classes
  $[P]$, on which the intersection form takes values in $\Z$.
\end{enumerate}
\end{theorem}

\begin{proof}
\textup{(i)} Write $\Phi$ for the matrix of $\phi$ and $G$ for the
(invertible, by Fact~\ref{fact:PD}) matrix of $\langle\,,\,\rangle$ in
any basis of $H^{2r-2}(X,\Ql(r-1))$. By Definition~\ref{def:phi},
$M=\Phi^{\mathsf T}G\,\Phi$, whence
\begin{equation}
\label{eq:sandwich}
  \mathrm{rank}\,M\;\le\;\mathrm{rank}\,\Phi=\dim_{\Ql}V
  \;\le\;\dim_{\Ql}H^{2r-2}\bigl(X,\Ql(r-1)\bigr)=\beta_{2r-2}(X).
\end{equation}
The two ends of \eqref{eq:sandwich} coincide: the left-hand side is
$1+m_{r}$ by Theorem~\ref{thm:spectrum}, and the right-hand side is
$\frac{q(q^{2r-1}+1)}{q+1}+1=1+m_{r}$ by \eqref{eq:bettiherm}.
Therefore equality holds throughout, so $\dim V=\beta_{2r-2}(X)$ and
$V=H^{2r-2}(X,\Ql(r-1))$.

\textup{(ii)} Equality of the first two terms of \eqref{eq:sandwich}
says $\mathrm{rank}\,M=\mathrm{rank}\,\Phi$, i.e.\
$\ker\phi=\ker M=W'$ by Theorem~\ref{thm:spectrum}.
Since $M$ is the pullback of a perfect pairing along a surjection,
$W'$ is also the radical of that pullback. Orthogonality of the
three summands is Lemma~\ref{lem:Wprime}.

\textup{(iii)} By (ii), $\phi$ restricts to an isomorphism
$W_0\oplus\Wst\xrightarrow{\sim}H^{2r-2}(X,\Ql(r-1))$; in particular
$\dim\phi(\Wst)=m_{r}$ and $\dim\phi(W_0)=1$.

We first show $\phi(\Wst)=\mathcal P^{2r-2}(X)$. Let $u=\sum_Pu_Pe_P\in\Wst$.
Since $\Wst\perp W_0=\Ql\mathbf 1$ in the standard inner product, we
have $\sum_Pu_P=0$, and therefore, by Lemma~\ref{lem:eig0}(i),
\[
  \bigl\langle\phi(u),L^{r-1}\bigr\rangle
  \;=\;\sum_{P}u_P\bigl\langle[P],L^{r-1}\bigr\rangle
  \;=\;\sum_Pu_P\;=\;0 .
\]
Hence $\phi(\Wst)\subseteq(L^{r-1})^{\perp}=\mathcal P^{2r-2}(X)$ by
Proposition~\ref{prop:lefschetz}(iii). Both spaces have dimension
$m_{r}$ --- the left as just noted, the right by
Proposition~\ref{prop:lefschetz}(iv) --- so the inclusion is an
equality.

Now $\phi(W_0)$ is a complement of $\phi(\Wst)=\mathcal P^{2r-2}(X)$ in
$H^{2r-2}(X,\Ql(r-1))$, and it is orthogonal to it: for $v\in W_0$, $u\in\Wst$,
\[
  \bigl\langle\phi(v),\phi(u)\bigr\rangle
  \;=\;v^{\mathsf T}Mu\;=\;\mu_{r}\,v^{\mathsf T}u\;=\;0 ,
\]
since $M$ acts on $\Wst$ by $\mu_{r}$ and $W_0\perp\Wst$. By
Proposition~\ref{prop:lefschetz}(iv) the orthogonal complement of
$\mathcal P^{2r-2}(X)$ is $\Ql L^{r-1}$, so $\phi(W_0)=\Ql L^{r-1}$. Pairing
$\phi(\mathbf 1)=t\,L^{r-1}$ with $L^{r-1}$ and using
Lemma~\ref{lem:eig0}(i) and Proposition~\ref{prop:lefschetz}(ii) gives
$t(q+1)=N_r$, i.e.
\[
  \sum_{P\in\mathcal G}[P]\;=\;\frac{N_r}{q+1}\,L^{r-1}.
\]

\textup{(iv)} For $u,u'\in\Wst$ we have
$\langle\phi(u),\phi(u')\rangle=u^{\mathsf T}Mu'=\mu_{r}\,
u^{\mathsf T}u'$, since $M$ acts on $\Wst$ by the scalar $\mu_{r}$.

\textup{(v)} By (i) the composite $Z^{r-1}(X)\otimes\Ql\to H^{2r-2}$ is
surjective. Since the classes of algebraic cycles span $H^{2r-2}$ and the pairing
is perfect, a cycle is numerically trivial if and only if its class vanishes.
Moreover the intersection matrix of cycles representing a $\Q$-basis of
$\mathrm N^{r-1}(X)\otimes\Q$ is nondegenerate, so their classes are linearly
independent over $\Ql$. Hence the induced map
on $\mathrm N^{r-1}(X)\otimes\Ql$ is injective as well as surjective. The
intersection numbers $P\cdot P'=f(\dim P\cap P')$ lie in $\Z$ by
Lemma~\ref{lem:f}, so the $\Q$-span of the $[P]$ is a $\Q$-structure on
which the form is integral.
\end{proof}

\begin{lemma}[The decomposition is defined over $\Q$]
\label{lem:qform}
Put
\[
  U_0:=W_0\cap\Q^{\mathcal G},\qquad
  \Ur:=\Wst\cap\Q^{\mathcal G},\qquad
  U':=W'\cap\Q^{\mathcal G}.
\]
Then each of these is a $\Q$-form of the corresponding
$\Ql$-subspace, that is
$W_0=U_0\otimes_{\Q}\Ql$, $\Wst=\Ur\otimes_{\Q}\Ql$ and
$W'=U'\otimes_{\Q}\Ql$; moreover
\[
  \Q^{\mathcal G}\;=\;U_0\ \perp\ \Ur\ \perp\ U',
  \qquad \dim_{\Q}\Ur=m_r ,
\]
and $\phi$ restricts to an isomorphism
$\Ur\xrightarrow{\ \sim\ }\mathcal P^{2r-2}(X)_{\Q}$.
\end{lemma}

\begin{proof}
As in the proof of Lemma~\ref{lem:Wprime}, $W_0$ and $\Wst$ are the
$\Ql$-spans of $U_0=\Q\mathbf 1$ and $\Ur=\ker_\Q(A-\theta_r)$; likewise
$W'=\ker M=U'\otimes\Ql$ (Theorem~\ref{thm:spectrum}), as $M$ is an integer matrix (Lemma~\ref{lem:f}).

The three subspaces of $\Q^{\mathcal G}$ are pairwise orthogonal, being
the $\Q$-points of pairwise orthogonal subspaces of $\Ql^{\mathcal G}$,
hence independent; and their $\Q$-dimensions are $1$, $m_r$ and
$N_r-1-m_r$ by Theorem~\ref{thm:spectrum}, summing to
$N_r=\dim_{\Q}\Q^{\mathcal G}$. So the decomposition holds.

Finally, $\phi(\Ur)\subseteq \mathcal P^{2r-2}(X)_{\Q}$: indeed
$\phi(\Ur)\subseteq\phi(\Wst)=\mathcal P^{2r-2}(X)$ by Theorem~\ref{thm:decomp}\textup{(iii)}, and
$\phi(\Ur)\subseteq\mathrm{span}_{\Q}\{[P]\}$ because $\Ur$ consists of
vectors with rational coordinates. Conversely let
$x\in \mathcal P^{2r-2}(X)_{\Q}$ and write $x=\phi(u)$ with
$u\in\Q^{\mathcal G}$. Decomposing $u=u_0+u_r+u'$ as above and
using $\phi(u')=0$ gives $x=\phi(u_0)+\phi(u_r)$ with
$\phi(u_0)\in\Q L^{r-1}$ and $\phi(u_r)\in \mathcal P^{2r-2}(X)$; since $x$ is
primitive and the sum $\Ql L^{r-1}\oplus \mathcal P^{2r-2}(X)$ is direct, we get
$\phi(u_0)=0$ and hence $x=\phi(u_r)\in\phi(\Ur)$. Injectivity is
inherited from Theorem~\ref{thm:decomp}\textup{(iii)}.
\end{proof}

\subsection{Consequence: the Tate conjecture}
\label{sec:tate-cor}

By Theorem~\ref{thm:decomp}, the whole cohomology of $X$ is spanned by
classes of cycles defined over $\F_{q^2}$. We deduce the Tate
conjecture for $X$, and compare the result with a direct count of
rational points.

\begin{theorem}
\label{thm:tate}
Let $r\ge2$, let $q$ be any prime power and $X=\Herm(2r-1,q^2)$. Then
the geometric Frobenius $\Frob_{q^{2}}$ acts on
$H^{2r-2}(X_{\overline\F_p},\Ql)$ by the scalar $q^{2r-2}$, and for
every $k$ the group $H^{2k}(X_{\overline\F_p},\Ql(k))$ is spanned by
algebraic cycle classes. In particular
Conjecture~\textup{\ref{conj:tate}} holds for $X$ unconditionally, and
the Tate conjecture also holds for $X_{K'}$ over every finitely generated
extension $K'$ of $\F_{q^{2}}$, since the cycles above are defined over
$\F_{q^{2}}$ and $H^{2k}(X_{\overline{K'}},\Ql(k))\cong
H^{2k}(X_{\overline\F_p},\Ql(k))$ by smooth and proper base change.
\end{theorem}

\begin{proof}
Every $P\in\mathcal G$ is defined over $\F_{q^{2}}$
(Lemma~\ref{lem:contained}(ii)), so $\Frob_{q^{2}}$ fixes each $P\in\mathcal G$
and hence each class $[P]\in H^{2r-2}(X,\Ql(r-1))$. By
Theorem~\ref{thm:decomp}(i) these classes span, so $\Frob_{q^{2}}$
acts as the identity on $H^{2r-2}(X,\Ql(r-1))$, i.e.\ by $q^{2r-2}$ on
$H^{2r-2}(X,\Ql)$. Hence every class in middle degree is a Tate class
and all are algebraic. For $k\ne r-1$ the statement is trivial, since
$H^{2k}(X,\Ql)\cong H^{2k}(\mathbb P^{2r-1},\Ql)$ is spanned by $L^{k}$ by
Proposition~\ref{prop:lefschetz}(i).
\end{proof}

\section{The standard conjecture of Hodge type for $\Herm(2r-1,q^2)$}
\label{sec:hodge}

In this section we prove the main theorem. We first make
Conjecture~\ref{conj:hdg} explicit for $X$, then show that $\phi$
identifies the form on the primitive part with a constant multiple of
the standard inner product, the constant being the eigenvalue $\mu_r$
of Theorem~\ref{thm:spectrum}, and deduce the conjecture.

\subsection{The statement for $X$}

For $Y=X=\Herm(2r-1,q^{2})$ we have $n=2r-2$, and in
Conjecture~\ref{conj:hdg} the only interesting case is $i=r-1$, where
$n-2i=0$, so that the form in question is
$(-1)^{r-1}\langle\,,\,\rangle$ itself; by
Theorem~\ref{thm:decomp}(i) the algebraic classes are all of
$H^{2r-2}(X,\Ql(r-1))$, and by Proposition~\ref{prop:lefschetz}(iv)
the primitive
part is $\mathcal P^{2r-2}(X)$. The relevant $\Q$-structure, in the sense of
\S\ref{subsec:hodge-survey}, is
\begin{equation}
\label{eq:Qstructure}
  \mathcal P^{2r-2}(X)_{\Q}\;:=\;\mathcal P^{2r-2}(X)\,\cap\,
  \mathrm{span}_{\Q}\bigl\{[P]\;:\;P\in\mathcal G\bigr\},
\end{equation}
on which $\langle\,,\,\rangle$ takes values in $\Z$ by
Theorem~\ref{thm:decomp}(v), and which spans $\mathcal P^{2r-2}(X)$ over
$\Ql$. So the conjecture takes the following form.

\begin{conjecture}[Standard conjecture of Hodge type for
$\Herm(2r-1,q^{2})$]
\label{conj:std}
Let $X=\Herm(2r-1,q^{2})$. The symmetric bilinear form
$(x,y)\mapsto(-1)^{r-1}\langle x,y\rangle$ is positive definite on
the $\Q$-vector space $\mathcal P^{2r-2}(X)_{\Q}$ of \eqref{eq:Qstructure}.
\end{conjecture}

\begin{proposition}[The full space of algebraic classes]
\label{lem:Qstructure}
$\mathcal A^{r-1}(X)=\mathrm{span}_{\Q}\{[P]:P\in\mathcal G\}$, and
consequently
\[
  \mathcal A^{r-1}_{\mathrm{prim}}(X)\;=\;\mathcal P^{2r-2}(X)_{\Q} .
\]
So Conjecture~\ref{conj:std} is exactly Conjecture~\ref{conj:hdg} for
$X$, with no ambiguity about which $\Q$-structure is meant.
\end{proposition}

\begin{proof}
Write $V_{\Q}=\mathrm{span}_{\Q}\{[P]\}$; clearly
$V_{\Q}\subseteq\mathcal A^{r-1}(X)$. By
Theorem~\ref{thm:decomp}(i),(v), $V_{\Q}$ has a $\Q$-basis
$v_1,\dots,v_{\beta}$, where $\beta=\beta_{2r-2}(X)$, which is also a
$\Ql$-basis of $H^{2r-2}(X,\Ql(r-1))$, and the Gram matrix
$G=\bigl(\langle v_i,v_j\rangle\bigr)$ has rational entries and is
invertible, the pairing being perfect (Fact~\ref{fact:PD}).

Let $\alpha\in\mathcal A^{r-1}(X)$ and write $\alpha=\sum_i\lambda_iv_i$
with $\lambda_i\in\Ql$. For each $j$, $\langle\alpha,v_j\rangle$ is an
intersection number of algebraic cycles, hence rational; so the vector
$G\lambda$ has rational entries, and $\lambda=G^{-1}(G\lambda)\in\Q^{\beta}$. Thus
$\alpha\in V_{\Q}$, proving the first claim. Intersecting with
$\mathcal P^{2r-2}(X)$ gives the second, since
$\mathcal A^{r-1}_{\mathrm{prim}}(X)=\mathcal A^{r-1}(X)\cap
\mathcal P^{2r-2}(X)$ by definition.
\end{proof}

\subsection{The proof}

We now prove the main result: combining the spectral
computation of Section~\ref{sec:decomp} with the cycle map, we show
that the form on $\mathcal P^{2r-2}(X)_{\Q}$ is a positive multiple of a
euclidean inner product.

\begin{theorem}[Main theorem]
\label{thm:main}
Let $r\ge2$, let $q$ be any prime power and set $X=\Herm(2r-1,q^{2})$.
Put
\[
  \kappa\;:=\;q^{\,r-1}N_{r-1}
  \;=\;q^{\,r-1}\prod_{i=1}^{r-1}\bigl(q^{2i-1}+1\bigr)
  \;\in\;\Z_{>0},
\]
let $\Q^{\mathcal G}\subset\Ql^{\mathcal G}$ carry the standard inner
product $\langle u,u'\rangle_{\mathrm{std}}=\sum_{P}u_Pu'_P$, and let
$\Ur=\Wst\cap\Q^{\mathcal G}\subset\Q^{\mathcal G}$
be the $\Q$-form \textup{(}Lemma~\textup{\ref{lem:qform}}\textup{)} of the
eigenspace $\Wst$ of Theorem~\textup{\ref{thm:bottomrow}}
\textup{(}see Lemma~\textup{\ref{rem:thetastarmin}}\textup{)}. Then $\phi$
restricts to an isomorphism of $\Q$-vector spaces
\[
  \psi\;:=\;\phi|_{\Ur}\;:\;\Ur\;\xrightarrow{\ \sim\ }\;
  \mathcal P^{2r-2}(X)_{\Q},
\]
and this isomorphism multiplies the form by the single constant
$\kappa$\textup{:}
\begin{equation}
\label{eq:mainiso}
  (-1)^{r-1}\bigl\langle\psi(u),\psi(u')\bigr\rangle
  \;=\;\kappa\cdot\langle u,u'\rangle_{\mathrm{std}}
  \qquad\text{for all }u,u'\in\Ur .
\end{equation}
Equivalently, let $u_1,\dots,u_{m_{r}}$ be any basis of $\Ur$, and
let $G=\bigl(\langle u_a,u_b\rangle_{\mathrm{std}}\bigr)_{a,b}$ be its
Gram matrix for the standard inner product. Then the Gram matrix of
the form
$(-1)^{r-1}\langle\,,\,\rangle$ in the corresponding basis
$\psi(u_1),\dots,\psi(u_{m_{r}})$ of $\mathcal P^{2r-2}(X)_{\Q}$ is exactly
$\kappa\,G$.
\end{theorem}

\begin{proof}
By Lemma~\ref{lem:qform} the map $\psi=\phi|_{\Ur}$ is an
isomorphism of $\Ur$ onto $\mathcal P^{2r-2}(X)_{\Q}$. By Theorem~\ref{thm:decomp}(iv),
\[
  \bigl\langle\psi(u),\psi(u')\bigr\rangle
  \;=\;\mu_{r}\,\langle u,u'\rangle_{\mathrm{std}}
  \qquad (u,u'\in\Ur),
\]
and by Theorem~\ref{thm:spectrum}
$\mu_{r}=(-1)^{r-1}q^{\,r-1}N_{r-1}$, so that
$(-1)^{r-1}\mu_{r}=\kappa$. This is \eqref{eq:mainiso}, and the
reformulation in terms of Gram matrices is immediate.
\end{proof}

{\begin{corollary}
\label{cor:std}
Conjecture~\ref{conj:std} holds for every $r\ge2$ and every prime
power $q$\textup{:} the symmetric bilinear form
$(x,y)\mapsto(-1)^{r-1}\langle x,y\rangle$ is positive definite on
$\mathcal P^{2r-2}(X)_{\Q}$.
\end{corollary}

\begin{proof}
Every element of $\mathcal P^{2r-2}(X)_{\Q}$ is $\psi(u)$ for a unique
$u\in\Ur$, and by \eqref{eq:mainiso}
$(-1)^{r-1}\langle\psi(u),\psi(u)\rangle=\kappa\,\langle
u,u\rangle_{\mathrm{std}}$, which is positive for $u\ne0$ since
$\kappa>0$ and the standard inner product on $\Q^{\mathcal G}$ is
positive definite.
\end{proof}}

Theorem~\ref{thm:main} and Corollary~\ref{cor:std} together are Theorem~\ref{thm:main-intro} of the introduction; together
with Theorem~\ref{thm:tate} it also gives
Theorem~\ref{thm:tate-intro}.

\section{The trace identity for the lattice of generators}
\label{sec:gentrace}

We return to $X=\Herm(2r-1,q^2)$ for arbitrary $r\ge2$ and $q=p^\nu$, and write
it as the Fermat variety $\sum_{i=0}^{2r-1}x_i^{\,d}=0$, $d=q+1$
(see \eqref{eq:hermform-intro}). Let $\mathrm N=\mathrm N^{r-1}(X)$ be the group
of algebraic cycles of codimension $r-1$ on $X_{\overline\F_p}$ modulo numerical
equivalence. Let $\Lambda\subset\mathrm N$ be the subgroup generated by the classes of
the generators. By Theorem~\ref{thm:decomp}, $\Lambda$ has finite index in $\mathrm N$,
and its Gram matrix is $M$ (Proposition~\ref{prop:gram}).

The diagonal torus $S=\mu_d^{2r}/\mu_d$ acts on $X$ and on $\mathcal G$, and
$|S|=d^{2r-1}$ is prime to $p$. Let $\WO=W(\F_{q^2})$ and let
$\omega\colon\mu_d(\F_{q^2})\to\WO^\times$ be the Teichm\"uller lift. Since the
Teichm\"uller lift is multiplicative, $\omega$ is an injective group homomorphism
onto $\mu_d(\WO)$, i.e.\ a faithful character of $\mu_d(\F_{q^2})$. For
$a\in(\Z/d)^{2r}$ with $\sum a_i=0$ put $\chi_a(g)=\prod_i\omega(g_i)^{a_i}$. This is a
character of $\mu_d^{2r}$ that is trivial on the diagonal, hence a character of
$S$; every character of $S$ is of this form. Call $a$ \emph{primitive} if all
$a_i\ne0$. The $\chi_a$-isotypic part of
the primitive middle cohomology is one-dimensional for primitive $a$ and zero
otherwise \cite[Prop.~7.4]{DMOS}, \cite[Thm.~I(i)]{ShiodaMA1979}; here the
characters take values in $\Q(\mu_d)$, which we embed both in $\WO[1/p]$ and in
$\overline{\Q}_\ell$, and the multiplicity of a character in $\Lambda\otimes\Q(\mu_d)$ can be
computed after either embedding. These are
stated over $\mathbb C$. The proof in \cite{DMOS}, via the covering
$[x_i]\mapsto[x_i^{\,d}]$ of $X$ onto the hyperplane $\sum_iy_i=0$, and an Euler characteristic
count of the tamely ramified eigensheaves, applies verbatim to $\ell$-adic
cohomology in characteristic $p\nmid d$. The trivial part of the middle cohomology is spanned by
$L^{r-1}$, where $L$ is the hyperplane class (\S\ref{sec:intro}); thus
$L^{r-1}\in\mathrm N^{r-1}(X)$ is the class of a linear section of codimension
$r-1$. Hence
\[
  \Lambda\otimes\WO=\WO\,L^{r-1}\ \oplus\ \bigoplus_{a\ \mathrm{primitive}}\Lambda_a,
  \qquad \Lambda_a=e_a(\Lambda\otimes\WO),\quad\operatorname{rank}\Lambda_a=1,
\]
where $e_a=|S|^{-1}\sum_g\chi_a(g)^{-1}g$. Here $e_0[P]=d^{-1}L^{r-1}$, since $[P]\cdot L^{r-1}=1$, and
$(L^{r-1})^2=\deg X=d$ is a unit. Let $e_\Lambda(a)$ be the valuation of the ideal
$\langle\Lambda_a,\Lambda_{-a}\rangle\subset\WO$. For generators $P,P'$ put
\[
  F_a(P,P')=\sum_{g\in S}\chi_a(g)^{-1}M_{gP,P'}=|S|\,\bigl\langle e_a[P],e_{-a}[P']\bigr\rangle .
\]
Then $e_\Lambda(a)=\min_{P,P'}v_p(F_a(P,P'))$, and the $p$-adic elementary divisors of
$M$ are $1$ together with $\{p^{e_\Lambda(a)}\}_{a\ \mathrm{primitive}}$.

\begin{proposition}[trace identity]
\label{prop:gentrace}
For every primitive $a$,
\[
  \sum_{P\in\mathcal R}\frac{F_a(P,P)}{|\mathrm{Stab}_S(P)|}
  \;=\;-\frac{\eta_r}{q+1}\;=\;(-1)^{r+1}q^{\,r-1}N_{r-1},
\]
where $\mathcal R$ is a set of representatives of the $S$-orbits on $\mathcal G$.
\end{proposition}

\begin{proof}
Matrices $B=(B_{P,P'})$ indexed by $\mathcal G$ act on
$\Q(\mu_d)^{\mathcal G}=\{f\colon\mathcal G\to\Q(\mu_d)\}$ by
$(Bf)(P)=\sum_{P'}B_{P,P'}f(P')$. Let $V_a$ be the subspace of $f$ with
$f(gP)=\chi_a(g)f(P)$ for $g\in S$. It has the basis $\{f_P\}$, where $P$ runs over
the $P\in\mathcal R$ with $\chi_a|_{\mathrm{Stab}(P)}=1$, and $f_P(gP)=\chi_a(g)$, $f_P=0$
off the orbit of $P$. Since $S$ preserves $\dim(\widetilde P\cap\widetilde P')$, we have
$M'_{gP,gP'}=M'_{P,P'}$, hence $(M'f)(gP)=\chi_a(g)(M'f)(P)$ for $f\in V_a$; that is,
$M'$ preserves $V_a$.

By \eqref{eq:Msqr} and Theorem~\ref{thm:decomp}, $M'=\eta_r\pi$, where $\pi$ is the
orthogonal projection onto $\Wst$. The space $\Wst$ is isomorphic to the primitive
middle cohomology as an $S$-representation, via the equivariant map $\phi$. So
$\pi(V_a)$ is one-dimensional, and $\mathrm{tr}(M'|_{V_a})=\eta_r$.

As $f_R(P)=\delta_{R,P}$ for $R\in\mathcal R$, the diagonal coefficient of
$M'|_{V_a}$ at $f_P$ is
$(M'f_P)(P)=|\mathrm{Stab}(P)|^{-1}\sum_{g\in S}\chi_a(g)M'_{P,gP}$. This sum
vanishes automatically when $\chi_a|_{\mathrm{Stab}(P)}\ne1$, so we may sum over all of
$\mathcal R$. By \eqref{eq:Mreform}, $M'=J-(q+1)M$, and $\sum_g\chi_a(g)=0$. Hence
$\sum_g\chi_a(g)M'_{P,gP}=-(q+1)F_{-a}(P,P)$. Replacing $a$ by $-a$ gives the
claim.
\end{proof}

\begin{corollary}
\label{cor:gentrace}
Let $m=r-1=\frac12\dim X$.
\begin{enumerate}
\item[\textup{(i)}] $e_\Lambda(a)\le\min_Pv_p(F_a(P,P))\le\nu m$ for every primitive $a$.
\item[\textup{(ii)}] The $p$-part of the discriminant group $\Lambda^\vee/\Lambda$ is annihilated by
  $q^m$.
\item[\textup{(iii)}] $q^m\,\mathrm N^m(X)\subset\Lambda$: every algebraic cycle of codimension $m$,
  multiplied by $q^m$, is numerically equivalent to an integral combination of
  generators.
\item[\textup{(iv)}] $v_p[\mathrm N:\Lambda]\le\frac12\,\nu m\,(\beta_{2m}(X)-1)$.
\end{enumerate}
\end{corollary}

\begin{proof}
(i) The orders $|\mathrm{Stab}(P)|$ are prime to $p$, and
$v_p(q^{r-1}N_{r-1})=\nu m$ because $N_{r-1}=\prod_{i=1}^{r-1}(q^{2i-1}+1)\equiv1\pmod p$.
So some term of Proposition~\ref{prop:gentrace} has valuation at most $\nu m$, and
$e_\Lambda(a)\le v_p(F_a(P,P))$ since $F_a(P,P)\in|S|\langle\Lambda_a,\Lambda_{-a}\rangle$.

(ii) By the decomposition above,
$(\Lambda^\vee/\Lambda)\otimes\WO\cong\bigoplus_a\WO/p^{e_\Lambda(a)}$.

(iii) The intersection form is integral on $\mathrm N$, so
$\Lambda\subset\mathrm N\subset\mathrm N^\vee\subset\Lambda^\vee$. By (ii),
$q^m\Lambda^\vee\subset\Lambda\otimes\Z_{(p)}$. Since $\mathrm N/\Lambda$ is a finite
$p$-group \cite{Shimada2001}, $q^m\mathrm N\subset\Lambda$.

(iv) $2v_p[\mathrm N:\Lambda]=v_p\disc\Lambda-v_p\disc\mathrm N\le\sum_ae_\Lambda(a)$,
and there are $\beta_{2m}(X)-1$ primitive characters.
\end{proof}

\begin{remark}
\label{rem:gentrace}
The trace identity of Proposition~\ref{prop:gentrace} gives only the upper bound
$e_\Lambda(a)\le\nu m$ of Corollary~\ref{cor:gentrace}(i). Lower bounds for
$e_\Lambda(a)$ are of interest because they bear on the open question of when
$\Lambda=\mathrm N^m(X)$; we do not treat them here.
\end{remark}

\section{Varieties finitely covered by Hermitian varieties}
\label{sec:fermat}

The main theorem propagates, with no further computation, to any
variety receiving a finite surjective map from a Hermitian variety.
Supersingular Fermat varieties of even dimension are examples
(\S\ref{subsec:fermatcase}). For the Tate conjecture this reduction to the
Hermitian case goes back to Tate \cite[p.~102]{Tate1965}. Along finite surjective morphisms of degree prime
to $p$, the $p$-adic bounds of \S\ref{sec:gentrace} descend as well
(\S\ref{subsec:pdescent}).

\subsection{Descent along a finite surjective morphism}

We show that the standard conjecture of Hodge type descends along a
finite surjective morphism (Lemma~\ref{lem:descent}), and that on a
Hermitian variety it holds for every polarization, not only for the
hyperplane class treated in Section~\ref{sec:hodge}
(Lemmas~\ref{lem:scaling} and~\ref{lem:allpol}); together these give
Theorem~\ref{cor:quotient}.

\begin{lemma}[Hdg descends]
\label{lem:descent}
Let $\pi:Y\to Z$ be a finite surjective morphism of nonsingular
projective varieties over an algebraically
closed field, and let $L_Z$ be an ample class on $Z$. If
$\mathrm{Hdg}(Y)$ holds with respect to the class $\pi^{*}L_Z$, then
$\mathrm{Hdg}(Z)$ holds with respect to $L_Z$.
\end{lemma}

\begin{proof}
Since $\pi$ is finite surjective, $\pi_*\pi^{*}=\deg(\pi)\cdot
\mathrm{id}$ on $H^{2i}(Z,\Ql(i))$ for every $i$, so
$\pi^{*}:H^{2i}(Z,\Ql(i))\to H^{2i}(Y,\Ql(i))$ is injective. It
carries algebraic cycle classes to algebraic cycle classes, hence
$\pi^{*}\mathcal A^{i}(Z)\subseteq\mathcal A^{i}(Y)$. It also carries
primitive classes to primitive classes: if $L_Z^{\,n-2i+1}x=0$ then
$(\pi^{*}L_Z)^{n-2i+1}\pi^{*}x=\pi^{*}(L_Z^{\,n-2i+1}x)=0$. Finally,
by the projection formula, for $x,y\in H^{2i}(Z,\Ql(i))$
\[
  \deg_Y\bigl(\pi^{*}x\cupp\pi^{*}y\cupp(\pi^{*}L_Z)^{\,n-2i}\bigr)
  \;=\;\deg(\pi)\cdot
  \deg_Z\bigl(x\cupp y\cupp L_Z^{\,n-2i}\bigr).
\]
So $\pi^{*}$ embeds $\mathcal A^{i}_{\mathrm{prim}}(Z,L_Z)$ into
$\mathcal A^{i}_{\mathrm{prim}}(Y,\pi^{*}L_Z)$ and multiplies the
Hodge-type form by the positive integer $\deg(\pi)$. A positive
definite form restricts to a positive definite form on any subspace.
\end{proof}

The hypothesis of Lemma~\ref{lem:descent} refers to the class
$\pi^{*}L_Z$, which need not be the class for which one happens to
know $\mathrm{Hdg}(Y)$. Two elementary observations remove this
difficulty for the varieties at hand.

\begin{lemma}[Scaling]
\label{lem:scaling}
Let $L$ be an ample class on a nonsingular projective $Y$ and let $e$
be a positive integer. Then $\mathrm{Hdg}(Y,L)$ and
$\mathrm{Hdg}(Y,eL)$ are equivalent.
\end{lemma}

\begin{proof}
$\ker\bigl((eL)^{k}\bigr)=\ker(L^{k})$ for every $k$, so the primitive
parts coincide; and the Hodge-type form for $eL$ is $e^{\,n-2i}>0$
times the one for $L$, so definiteness is unaffected.
\end{proof}

\begin{lemma}[Every polarization of a Hermitian variety is covered]
\label{lem:allpol}
Let $m\ge1$, let $q$ be a prime power, and put
$Y=\Herm(2m+1,q^{2})$. Then $\mathrm{Hdg}(Y,L')$ holds for
\emph{every} ample class $L'$ on $Y$.
\end{lemma}

\begin{proof}
Suppose first $m\ge2$, so $\dim Y=2m\ge4$, in
particular $\ge3$. By the
Grothendieck--Lefschetz theorem for Picard groups
\textup{\cite[Exp.~XII, Cor.~3.7]{SGA2}}, a nonsingular hypersurface
of dimension $\ge3$ in $\mathbb P^{n+1}$ over an algebraically closed
field has $\mathrm{Pic}\cong\Z$, generated by the hyperplane class
$L$. Hence every ample class on $Y$ is $eL$ for some integer $e\ge1$.
Corollary~\ref{cor:std} gives $\mathrm{Hdg}(Y,L)$ in the middle degree
and Proposition~\ref{prop:belowmiddle}(i) gives it in every degree below
the middle, so $\mathrm{Hdg}(Y,L)$ holds; Lemma~\ref{lem:scaling} then
gives $\mathrm{Hdg}(Y,eL)$.

If $m=1$ then $Y$ is a nonsingular projective surface, and for a
surface the standard conjecture of Hodge type with respect to any
ample class is the Hodge index theorem
\textup{\cite[Thm.~V.1.9]{Hartshorne}}, which holds unconditionally.
\end{proof}

\begin{theorem}[Finite covers by Hermitian varieties]
\label{cor:quotient}
Let $m\ge1$, let $q$ be a prime power, and let $Z$ be a nonsingular
projective variety of dimension $2m$ over $\overline{\F}_p$ admitting
a finite surjective morphism
\[
  \pi\;:\;\Herm(2m+1,q^{2})\;\longrightarrow\;Z .
\]
Then the standard conjecture of Hodge type holds for $Z$, with respect
to every ample class on $Z$. In particular this applies whenever
$Z=\Herm(2m+1,q^{2})/G$ is a nonsingular quotient by a finite group
$G$ of automorphisms.
\end{theorem}

\begin{proof}
Let $L_Z$ be ample on $Z$. Since $\pi$ is finite, $\pi^{*}L_Z$ is
ample on $\Herm(2m+1,q^{2})$, so Lemma~\ref{lem:allpol} supplies the
hypothesis of Lemma~\ref{lem:descent}, which gives
$\mathrm{Hdg}(Z,L_Z)$. For the last sentence, the quotient morphism
onto a nonsingular quotient is finite and surjective.
\end{proof}

\subsection{Supersingular Fermat varieties}
\label{subsec:fermatcase}

Supersingular Fermat varieties of even dimension are finite quotients
of Hermitian varieties, so Theorem~\ref{cor:quotient} applies to them.
We first fix the terminology and recall which Fermat varieties are
involved; neither is used anywhere before this point.

\subsubsection*{Supersingularity: which definition}
\label{subsec:ss}

\begin{definition}[Shioda supersingularity;
{\cite[Def.~2.8]{SK1979}}]
\label{def:ss}
A nonsingular projective variety $Y$ of even dimension $n$ over an
algebraically closed field is \emph{supersingular} if $H^{2i}(Y,\Ql(i))$
is spanned by classes of algebraic cycles for every $i$ with
$0\le i\le n$. 
\end{definition}

\subsubsection*{Which Fermat varieties are reached}
\label{subsec:fermat}

For the Fermat varieties $X_N^{n}\subset\mathbb P^{n+1}$ over $\overline{\F}_p$ with $p\nmid N$, the Tate
conjecture is a classical theorem in the following case.

\begin{theorem}[Shioda--Katsura {\cite[Thms.~2.10, 4.1]{SK1979}}; see also \cite{Shioda1979}]
\label{thm:sk}
If there exists $\nu\ge1$ with $p^{\nu}\equiv-1\pmod N$, then the
Fermat hypersurface of degree $N$ and even dimension
is unirational, its primitive
$\ell$-adic cohomology is entirely spanned by algebraic cycle classes,
and in particular the Tate conjecture holds for it unconditionally.
\end{theorem}

The Hermitian varieties of
\S\S\ref{sec:setup}--\ref{sec:hodge} are precisely the \emph{extremal}
case $N=p^{\nu}+1$ of this criterion, and the general case is deduced
from them below.

\begin{lemma}
\label{lem:cover}
Let $N\ge2$ with $p\nmid N$, let $q=p^{\nu}$ be such that $N\mid q+1$,
and set $e=(q+1)/N$. Let $\mu_e^{2m+2}/\mu_e$ act on
$\mathbb P^{2m+1}$ by $[x_i]\mapsto[\zeta_ix_i]$ with
$\zeta_i^{\,e}=1$. Then this action preserves $\Herm(2m+1,q^{2})$, and
the quotient map is
\[
  \pi\;:\;\Herm(2m+1,q^{2})\;\longrightarrow\;X_N^{2m},
  \qquad
  [x_0:\cdots:x_{2m+1}]\;\longmapsto\;[x_0^{\,e}:\cdots:x_{2m+1}^{\,e}] ,
\]
a finite surjective morphism of nonsingular projective varieties of
dimension $2m$, with $\pi^{*}\mathcal O_{X_N}(1)=\mathcal
O_{\Herm}(e)$ and $\deg(\pi)=e^{\,2m+1}$. In particular $X_N^{2m}$ is
a nonsingular finite quotient of $\Herm(2m+1,q^{2})$.
\end{lemma}

\begin{proof}
Since $e\mid q+1$, we have $\zeta_i^{\,q+1}=1$, so the action
preserves $\sum_ix_i^{\,q+1}$ and hence $\Herm(2m+1,q^{2})$; the
invariants of $\mu_e^{2m+2}/\mu_e$ on $\mathbb P^{2m+1}$ are generated
by the $x_i^{\,e}$, so the quotient map is $[x_i]\mapsto[x_i^{\,e}]$.

This assignment is in any case a morphism
$\mathbb P^{2m+1}\to\mathbb P^{2m+1}$, since the $x_i^{\,e}$ have no
common zero, and it is finite. If $\sum_ix_i^{\,q+1}=0$ then
$\sum_i(x_i^{\,e})^{N}=\sum_ix_i^{\,eN}=\sum_ix_i^{\,q+1}=0$, so it
maps $\Herm(2m+1,q^{2})$ into $X_N^{2m}$. Both are nonsingular of
dimension $2m$ (here $p\nmid N$ and $p\nmid q+1$), and $X_N^{2m}$ is
irreducible, so the image, being closed of dimension $2m$, is all of
$X_N^{2m}$. The map is given by forms of degree $e$, whence
$\pi^{*}\mathcal O(1)=\mathcal O(e)$, and the projection formula gives
$\deg(\pi)\cdot N=\deg(\pi)\,\deg(L_Z^{2m})=\deg\bigl((\pi^{*}L_Z)^{2m}\bigr)
=e^{2m}\deg(L_Y^{2m})=e^{2m}(q+1)$, so $\deg(\pi)=e^{2m}(q+1)/N=e^{2m+1}$.
\end{proof}

\begin{corollary}[Standard conjecture of Hodge type for supersingular
Fermat varieties]
\label{thm:fermat}
Let $p$ be a prime, $m\ge1$, and let $N\ge2$ satisfy
\[
  p^{\nu}\equiv-1 \pmod N \qquad\text{for some }\nu\ge1 .
\]
Then the standard conjecture of Hodge type holds for the nonsingular
Fermat variety $X_N^{2m}\subset\mathbb P^{2m+1}$ over
$\overline{\F}_p$.
\end{corollary}

\begin{proof}
Set $q=p^{\nu}$, so that $N\mid q+1$ and $p\nmid N$. By
Lemma~\ref{lem:cover} there is a finite surjective morphism
$\Herm(2m+1,q^{2})\to X_N^{2m}$ of nonsingular projective varieties of
dimension $2m$, so Theorem~\ref{cor:quotient} applies.
\end{proof}

\subsection{$p$-adic descent}
\label{subsec:pdescent}

Let $X=\Herm(2m+1,q^2)$, $q=p^\nu$, let $\Lambda\subset\mathrm N^m(X)$ be the lattice of
generators (\S\ref{sec:gentrace}), and let $\pi\colon X\to Z$ be a finite surjective
morphism onto a nonsingular projective variety $Z$. Such a $\pi$ is flat, so
$\pi^*$ and $\pi_*$ are defined on cycles modulo numerical equivalence. They
satisfy $\pi_*\pi^*=\deg\pi$ and the projection formula
$\langle\pi^*x,y\rangle_X=\langle x,\pi_*y\rangle_Z$. Let
\[
  \Lambda_Z=\bigl\langle\pi_*[P]:P\in\mathcal G\bigr\rangle\subset\mathrm N^m(Z).
\]

\begin{proposition}
\label{prop:pdescent}
\begin{enumerate}
\item[\textup{(i)}] $q^m\deg(\pi)\cdot\mathrm N^m(Z)\subset\Lambda_Z$, and $\Lambda_Z^\vee/\Lambda_Z$ is
  annihilated by $q^m\deg(\pi)$. In particular, if $p\nmid\deg\pi$, then the
  $p$-part of $\Lambda_Z^\vee/\Lambda_Z$ is annihilated by $q^m$, and
  $q^m\,\mathrm N^m(Z)\subset\Lambda_Z\otimes\Z_{(p)}$.
\item[\textup{(ii)}] For every prime $\ell\ne p$ with $\ell\nmid\deg\pi$, the lattice
  $\Lambda_Z\otimes\Z_\ell=\mathrm N^m(Z)\otimes\Z_\ell$ is unimodular.
\end{enumerate}
\end{proposition}

\begin{proof}
By \cite[Thm.~1.1]{Shimada2001}, $\Lambda^\vee/\Lambda$ is a $p$-group. So
Corollary~\ref{cor:gentrace}(ii),(iii) give $q^m\Lambda^\vee\subset\Lambda$ and
$q^m\mathrm N^m(X)\subset\Lambda$, and $\Lambda\otimes\Z_\ell$ is unimodular for
$\ell\ne p$.

(i) Let $y\in\mathrm N^m(Z)$. Then $q^m\pi^*y\in\Lambda$, and applying $\pi_*$ gives
$q^m\deg(\pi)\,y\in\pi_*\Lambda=\Lambda_Z$. Let $x\in\Lambda_Z^\vee$. For
$\lambda\in\Lambda$ we have $\langle\pi^*x,\lambda\rangle_X=\langle x,\pi_*\lambda\rangle_Z\in\Z$, so
$\pi^*x\in\Lambda^\vee$. Hence $q^m\pi^*x\in\Lambda$ and $q^m\deg(\pi)\,x\in\Lambda_Z$.

(ii) Over $\Z_\ell$ the same argument, with $\Lambda^\vee\otimes\Z_\ell=\Lambda\otimes\Z_\ell$,
gives $\deg(\pi)(\Lambda_Z^\vee\otimes\Z_\ell)\subset\Lambda_Z\otimes\Z_\ell$. As
$\deg\pi$ is an $\ell$-adic unit, $\Lambda_Z\otimes\Z_\ell$ is unimodular. Since
$\Lambda_Z\subset\mathrm N^m(Z)\subset\mathrm N^m(Z)^\vee\subset\Lambda_Z^\vee$, all four
coincide $\ell$-adically.
\end{proof}

When $\pi$ is the quotient by a group of automorphisms that permutes the
generators, the descent is more precise.

\begin{corollary}
\label{cor:pdescent}
Let $T\subset\PGU_{2m+2}(q)$ be a finite subgroup such that $Z=X/T$ is
nonsingular, and let $\pi\colon X\to Z$ be the quotient map, so $\deg\pi=|T|$.
\begin{enumerate}
\item[\textup{(i)}] $\langle\pi_*P,\pi_*P'\rangle_Z=\sum_{t\in T}M_{P,tP'}$ for $P,P'\in\mathcal G$.
\item[\textup{(ii)}] If $p\nmid|T|$, then $\pi^*$ induces isomorphisms
  $\Lambda_Z\otimes\Z_p\cong(\Lambda\otimes\Z_p)^T$ and
  $\mathrm N^m(Z)\otimes\Z_p\cong(\mathrm N^m(X)\otimes\Z_p)^T$, which multiply the
  intersection form by the unit $|T|$.
\item[\textup{(iii)}] Let $L_Z$ be an ample class on $Z$ with $\pi^*L_Z\in\Q_{>0}L$ (this is
  automatic if $m\ge2$, and holds for the Fermat quotients below). Under $\pi^*$, the primitive part of $\mathcal A^m(Z)$ with respect to $L_Z$ is identified with
  the $T$-invariants of $\mathcal P^{2m}(X)_\Q\cong\Ur$. On it,
  $(-1)^m\langle\,,\rangle_Z$ is $|T|^{-1}q^mN_m$ times $\langle\,,\rangle_{\mathrm{std}}$.
\end{enumerate}
For $Z=X_N^{2m}$ with $N\mid q+1$, take $T=\mu_e^{2m+2}/\mu_e\subset S$ with
$e=(q+1)/N$ (Lemma~\ref{lem:cover}); here $|T|=e^{2m+1}$ is prime to $p$.
\end{corollary}

\begin{proof}
Now $\pi^*\pi_*=\sum_{t\in T}t_*$, and $T$ permutes $\mathcal G$ isometrically.

(i) $\langle\pi_*P,\pi_*P'\rangle_Z=|T|^{-1}\sum_{t,t'}\langle tP,t'P'\rangle_X
=\sum_t\langle P,tP'\rangle_X$.

(ii) With $e_T=|T|^{-1}\sum_tt\in\Z_p[T]$ we get
$\pi^*(\Lambda_Z\otimes\Z_p)=\pi^*\pi_*(\Lambda\otimes\Z_p)=e_T(\Lambda\otimes\Z_p)
=(\Lambda\otimes\Z_p)^T$. Moreover $\pi^*$ is injective because $\pi_*\pi^*=|T|$ is a
unit. The same argument applies to $\mathrm N^m$.

(iii) The map $\pi^*$ preserves primitivity (Lemma~\ref{lem:descent}), and the
isomorphism $\psi$ of Theorem~\ref{thm:main-intro} is equivariant for the
automorphisms of $X$ preserving $\mathcal G$. Now use the projection formula.
\end{proof}

\begin{remark}
\label{rem:autPGU}
The group of linear automorphisms of $X$, i.e.\ of the projective
transformations of $\mathbb P^{2m+1}$ preserving $X$, is $\PGU_{2m+2}(q)$
\cite{Shioda1988}, \cite[Example~1]{Kontogeorgis2002}. If $(m,q)\ne(1,3)$, every
automorphism of $X$ is linear \cite{MM1964}, so $\mathrm{Aut}(X)=\PGU_{2m+2}(q)$.
\end{remark}

\begin{remark}
For primes $\ell\mid\deg\pi$ nothing is claimed: $\Lambda_Z\otimes\Z_\ell$ may be
smaller than $\mathrm N^m(Z)\otimes\Z_\ell$.
\end{remark}

\bigskip

\noindent{\bf Declaration:}
During the preparation of this work, the author made extensive use of Claude Opus 5 and Claude Opus 5.5 for formulating assertions, proving results, 
drafting the paper, refining the exposition, and correcting deficiencies in the writing. After using this tool, the author
reviewed and verified the content and takes full responsibility for
the correctness of the theorems and the content of the published
article.

\end{document}